\documentclass[12pt,centertags,oneside]{article}
\usepackage{amsmath,amstext,amsthm,amscd,typearea}
\usepackage{amssymb}
\usepackage{a4wide}
\usepackage[mathscr]{eucal}
\usepackage{mathrsfs}
\usepackage{typearea}
\usepackage{charter}
\usepackage{pdfsync}
\usepackage{url}
\usepackage[T1]{fontenc}

\usepackage{hyperref}

\usepackage{xcolor}

\usepackage[a4paper,width=16.2cm,top=3cm,bottom=3cm]{geometry}

\numberwithin{equation}{section}

\newtheorem{theorem}{Theorem}[section]
\newtheorem{definition}[theorem]{Definition}
\newtheorem{proposition}[theorem]{Proposition}
\newtheorem{corollary}[theorem]{Corollary}
\newtheorem{lemma}[theorem]{Lemma}
\newtheorem{remark}[theorem]{Remark}

\newcommand{\cali}[1]{\mathscr{#1}}

\newcommand{\supp}{{\rm Supp}}

\newcommand{\Leb}{\mathop{\mathrm{Leb}}\nolimits}

\newcommand{\dist}{\mathop{\mathrm{dist}}\nolimits}
\renewcommand{\Re}{\mathop{\mathrm{Re}}\nolimits}
\renewcommand{\Im}{\mathop{\mathrm{Im}}\nolimits}
\newcommand{\vol}{\mathop{\mathrm{vol}}}

\newcommand{\ddc}{dd^c}

\newcommand{\PSH}{{\rm PSH}}

\newcommand{\B}{\mathbb{B}}
\newcommand{\C}{\mathbb{C}}
\newcommand{\D}{\mathbb{D}}
\newcommand{\N}{\mathbb{N}}

\newcommand{\R}{\mathbb{R}}
\renewcommand\P{\mathbb{P}}

\renewcommand{\S}{\mathbb{S}}

\newcommand{\E}{{\mathbb{E}}}

\newcommand{\boldsym}[1]{\boldsymbol{#1}}

\title{\bf  Bergman kernel functions associated to measures supported on  totally real submanifolds}
\providecommand{\keywords}[1]{\textbf{\textit{Keywords:}} #1}
\providecommand{\subject}[1]{\textbf{\textit{Mathematics Subject Classification 2020:}} #1}

\author{George Marinescu and Duc-Viet Vu}

\newcommand{\Addresses}{{
		\bigskip
		\footnotesize
		
\textsc{George Marinescu, University of Cologne, Department of Mathematics and Computer Science, Division of Mathematics, Weyertal 86-90, 50931 K\"oln,  Germany.}
		\noindent
		\par\nopagebreak
		\noindent
		\textit{E-mail address}: \texttt{gmarines@math.uni-koeln.de}	
		\newline

		\textsc{Duc-Viet Vu, University of Cologne, Department of Mathematics and Computer Science, Weyertal 86-90, 50931 K\"oln,  Germany}
		\noindent
		\par\nopagebreak
		\noindent
		\textit{E-mail address}: \texttt{vuviet@math.uni-koeln.de}	
}}

\date{\today}
\begin{document}
\maketitle
\begin{abstract} We prove that the Bergman kernel function associated to  a smooth measure supported on a piecewise-smooth maximally totally real submanifold $K$ in $\C^n$ is of polynomial growth (e.g, in dimension one,  $K$ is a finite union of transverse Jordan arcs in $\C$). Our bounds are sharp when $K$ is smooth. We give an application to  equidistribution of zeros of random polynomials extending a result of Shiffman-Zelditch to the higher dimensional setting.  
\end{abstract}
\noindent
, {Bernstein-Markov measure}, {totally real submanifolds}, {equidistribution}, {random polynomial}.
\\

\noindent
\subject{32U15}, {32Q15}, {32V20}, {60F10}, {60B20}, {41A10}.

\tableofcontents

\section{Introduction}

Let $K$ be a non-pluripolar compact subset in $\C^n$, \emph{i.e,} $K$ is not contained in $\{\varphi = -\infty\}$ for some plurisubharmonic (psh) function $\varphi$ on $\C^n$. Let $\mu$ be a probability measure whose support is non-pluripolar and is contained in $K$, and $Q$ be a real continuous function on $K$. Let $\mathcal{P}_k$ be the space of restrictions to $K$ of complex polynomials of degree at most $k$ on $\C^n$. The  scalar product
$$\langle s_1, s_2\rangle_{L^2(\mu,k Q)}:= \int_K s_1  \overline s_2  e^{-2 k Q} d \mu$$
induces the $L^2(\mu, k Q)$-norm on $\mathcal{P}_k$.   \emph{The Bergman kernel function of order $k$} associated to $\mu$ with weight $Q$ is defined by 
$$B_k(x):=  \sup_{s \in \mathcal{P}_k} |s(x) e^{-k Q(x)}|^2/ \|s\|^2_{L^2(\mu, k Q)}$$
for $x \in K$. 
Equivalently if $(s_1, \ldots, s_{d_k})$ (here  $d_k$ denotes the dimension of $\mathcal{P}_k$) is an orthonormal basis of $\mathcal{P}_k$ with respect to the $L^2(\mu,k Q)$-norm, then 
$$B_k(x)= \sum_{j=1}^{d_k} |s_j(x)|^2 e^{-2k Q(x)}.$$ 
When $Q \equiv 0$, we say that $B_k$ is \emph{unweighted}. 
In this case ($Q \equiv 0$) the inverse of $B_k$ is known as \emph{the Christoffel function} in the literature on orthogonal polynomials.  
In practice we also use a modified version of the Bergman kernel function as follows:
$$\tilde{B}_k(x):= \sup_{s \in \mathcal{P}_k} |s(x)|^2/ \|s\|^2_{L^2(\mu, k Q)}$$
for $x \in \C^n$. The advantage of $\tilde{B}_k$ is that it is well defined on $\C^n$.  
  
The asymptotics of the Bergman kernel function 
(or its inverse, the Christoffel function) is essential for many applications in 
(higher dimensional or not) real analysis including approximation theory, 
random matrix theory, etc. There is an immense literature on such asymptotics. 
We refer to \cite{Beckermann-Putinar-Saff-Sty,BoucksomBermanWitt,
Bloom-Levenberg-PW,Danka-Totik,Deift-book,Dunkl-Xu-book,Kroo-Lubinsky,
Lubinsky,Totik,Xu-simplex}, to cite a just few,
for an overview on this very active research field.

Most standard settings are measures supported on concrete domains 
on $\R^n\subset\C^n$  (such as balls or simplexes in $\R^n$) 
or in the unit ball in $\C^n$. Considering measures on $\C^n$ 
whose support are not necessarily in $\R^n$ are also important 
in many applications; e.g., one can consult 
\cite{Beckermann-Putinar-Saff-Sty,Gustafsson-Putinar-Saff-Stylianopoulos,
Totik-transaction,Totik} where the authors consider measures supported 
on finite unions of piecewise smooth Jordan curves $\C$ or 
domains in $\C$ bounded by Jordan curves. 
We refer to the end of this section for a concrete application 
to the equidistribution of zeros of random polynomials.

All of settings mentioned above are particular cases of a more natural situation where our measures are supported on piecewise-smooth domains in a generic Cauchy-Riemann submanifold $K$ in $\C^n$. This is the context in which we will work on in this paper.

We underline that in view of potential applications, 
it is important to work with piecewise-smooth compact sets $K$ 
(rather than only smooth ones). 
In what follows, by a (convex) \emph{polyhedron} in $\R^M$, 
we mean a subset in $\R^M$ which is the intersection of a 
finite number of closed half-hyperplanes in $\R^M$.

\begin{definition}\label{def_piecewisesmooth}
A subset $K$ of a real $M$-dimensional smooth manifold $Y$ is called a nondegenerate $\mathcal{C}^5$ piecewise-smooth submanifold of dimension $m$ if for every point $p \in K,$ there exists a local chart $(W_p,\Psi)$ of $Y$ such that $\Psi$ is a $\mathcal{C}^5$-diffeomorphism from $W_p$ to the unit ball of $\R^{M}$ and  $\Psi(K \cap W_p)$ is the intersection with the unit ball  of a finite union of  convex polyhedra of the same dimension $m.$ 
\end{definition}

A point $p \in K$ is said to be \emph{a regular point} 
of $K$ if the above local chart $(W_p, \Psi_p)$ can be 
chosen such that $\Psi_p(K\cap W_p)$ is the intersection 
of the unit ball with an $m$-dimensional vector subspace in $\R^M$, 
in other words, $K$ is a $m$-dimensional submanifold locally near $p$. 
The regular part of $K$ is the set of regular points of $K.$  
\emph{The singular part} of $K$ is the complement of the regular part 
of $K$ in $K.$ Hence if $K$ is a smooth manifold with boundary, 
then the boundary of $K$ is the singular part of $K$ and its complement 
in $K$ is the regular part of $K$. 

Now let $K$ be a nondegenerate $\mathcal{C}^5$ piecewise-smooth submanifold of $X.$ Since $X$ is a complex manifold, its real tangent spaces have a natural complex structure $J$. We say that  $K$ is \emph{(Cauchy-Riemann or CR for short) generic} in the sense of Cauchy-Riemann geometry if for every $p \in K$ and every sequence of regular points $(p_m)_m \subset K$ approaching to $p$, any limit space of the sequence of tangent spaces of $K$ at $p_m$ is not contained  in a complex hyperplane of the (real) tangent space at $p$ of $X$ (equivalently, if $E$ is a limit space of the sequence $(T_{p_m}K)_{m\in \N}$ of tangent spaces at $p_m$, then  we have $E+ J E= T_p X$, where $T_p X$ is the real tangent space of $X$ at $p$). 

For a CR generic $K$, note that the space $T_pK \cap J T_p K$ ($p$ is a regular point in $K$) is invariant under $J$ and hence has a complex structure induced by $J$. In this case, the complex dimension of $T_pK \cap J T_p K$ is the same  for every $p$ and is called \emph{the CR dimension} of $K$.  If $r$ denotes the CR dimension of $K$, then $r= \dim K -n$. Thus the dimension of a generic $K$ is at least $n.$ 

If $Y$ is CR generic and  $\dim Y=n$, then $Y$ is said to be (maximally) \emph{totally real}, and it is locally the graph of a smooth function over a small ball centered at $0 \in \R^n$ which is tangent at $0$ to $\R^n$. Examples of piecewise-smooth totally real submanifolds are  polygons in $\C$ or boundaries of polygons in $\C$, and polyhedra of dimension $n$ in $\R^n \subset \C^n$.

A notion playing an important role in the study of Bergman kernel functions is the following extremal function
$$V_{K,Q}:= \sup\{\psi \in \mathcal{L}(\C^n): \, \psi \le Q \quad \text{ on } K\},$$
where $\mathcal{L}(\C^n)$ is the set of psh functions $\psi$ on $\C^n$ such that $\psi(z)- \log |z|$ is bounded at infinity on $\C^n$.

Since $K$ is non-pluripolar, the upper semi-continuous regularisation 
$V_{K,Q}^*$ of $V_{K,Q}$ belongs to  $\mathcal{L}(\C^n)$. 
If $V_{K,Q}= V^*_{K,Q}$, then we say $(K,Q)$ is \emph{regular}. 
A stronger notion is the following: we say that $K$ is \emph{locally regular} 
if for every $z \in K$ there is an open neighborhood $U$ of $z$ in $\C^n$ 
such that for every increasing sequence of psh functions $(u_j)_j$ on $U$ 
with $u_j \le 0$ on $K \cap U$, then 
$$(\sup_j u_j)^* \le 0$$
on $K \cap U$. One can see that if $K$ is a locally regular set, 
then $(K,Q)$ is regular for every continuous function $Q$ on $K$.  
The following result answers the question raised in 
\cite[Remark 1.8]{BoucksomBermanWitt}.

\begin{theorem} \label{the-BMpropertylocallyreguCn}   Let $K$ be a compact generic Cauchy-Riemann nondegenerate $\mathcal{C}^5$ piecewise-smooth submanifold  in $\C^n$. Then $K$ is locally regular. 
\end{theorem}

Note that it was known that $K$ is locally regular if $K$ 
is smooth real analytic; see, e.g., \cite[Corollary 1.7]{BoucksomBermanWitt}. 

The measure $\mu$ is said to be a \emph{Bernstein-Markov measure} 
(with respect to $(K,Q)$) if  for every $\epsilon >0$, 
there exists $C>0$ such that  
\begin{align*}
\sup_K |s|^2 e^{- 2 k Q} \le C e^{\epsilon k} \|s\|^2_{L^2(\mu, kQ)}
\end{align*}
for every  $s \in \mathcal{P}_k$. In other words, 
the Bergman kernel function of order $k$ grows at most subexponentially, 
\emph{i.e.,}  $\sup_K B_k =O(e^{\epsilon k})$ as $k \to \infty$ 
for every $\epsilon>0$.  

For some examples of Bernstein-Markov measures and criteria 
checking this condition, we refer to  \cite{Bloom-Levenberg-PW}. 
However, apart from few explicit geometric situations, 
there are not many (geometric) examples of Bernstein-Markov 
measures in higher dimensions. This is the motivation for 
our next main result giving a large geometric class 
of Bernstein-Markov measures.  

\begin{theorem} \label{the-BMpropertycn}
Let $K$ be a compact generic Cauchy-Riemann nondegenerate 
$\mathcal{C}^5$ piecewise-smooth submanifold in $\C^n$. 
Let $\mu$ be a finite measure supported on $K$ such that there exist 
constants $\tau>0, r_0>0$ satisfying $\mu(\B(z,r)\cap K) \ge r^\tau$ 
for every $z \in K$,  $r \le r_0$  (where $\B(z,r)$ denotes the ball 
of radius $r$ centered at $z$ in $\C^n$). 
Then  for  every continuous function $Q$ on $K$, 
$\mu$ is a Bernstein-Markov measure with respect to $(K,Q)$.
\end{theorem}

To the best of our knowledge, the above result was only known 
when $K$ is real analytic. We are not aware of results of this type 
in previous literature for maximally totally real submanifolds. 
A measure $\mu_0$ on $K$ is said to be a smooth volume form 
on $K$ if $\mu_0$ is given by a smooth volume form
locally at every regular point in $K$, 
and if for every singular point $p$ in $K$, there is a local chart $(\Psi,W_p)$ 
as in Definition \ref{def_piecewisesmooth}, such that 
$\Psi(K \cap W_p)= \bigcup_{j=1}^s P_j$, where $P_j$'s 
are polyhedra in $\C^n$ of the same dimension for every $1 \le j \le s$, 
and  the restriction of $\mu_0$ to $P_j$ is a smooth volume form on 
$P_j$ for $1 \le j \le s$. 

Let $\Leb_K$ be now a smooth volume form on $K$. 
Then for any $M>0$ the measure $\mu=|z-z_0|^M \Leb_K$ 
satisfies the hypothesis of 
Theorem \ref{the-BMpropertycn}. Here is our next main result.

\begin{theorem} \label{the-strongBMintroCn}  
Let $K$ be a compact generic Cauchy-Riemann nondegenerate 
$\mathcal{C}^5$ piecewise-smooth submanifold in $\C^n$
of dimension $n_K$. 
Let $Q$ be a H\"older continuous function of H\"older exponent 
$\alpha \in (0,1)$ on $K$, and let $\Leb_K$ be a smooth volume form on $K$, 
and $\mu = \rho \Leb_K$, where $\rho \ge 0$ and 
$\rho^{-\lambda} \in L^1(\Leb_K)$ for some  constant $\lambda>0$. 
Then we have 
$$\sup_K  B_k \le C k^{2n_K(\lambda+1)/ (\alpha\lambda)}\,,$$
for some constant $C>0$ independent of $k$.
\end{theorem}

We would like to point out that \cite[Remark 3.2]{Berman-weighted-Cn} 
shows that the regularity of weights affects considerably the growth 
of the Bergman kernel function. One can also consult the last reference for upper bounds 
for $K=\C^n$ and $\mu= \Leb_{\C^n}$ (the Lebesgue measure on $\C^n$).  

There have been very few works concerning polynomial growth 
of Bergman kernel functions associated to measures supported on real submanifolds 
in higher dimensions.  With an exception of \cite{Berman-OrtegaCerda}, 
known upper bounds on $B_k$ were proved mostly based on special geometric structures 
of the compact $K \subset \R^n$ (see, e.g.,\cite{Kroo-Lubinsky,Lubinsky}). 
Such a method is not useful in dealing with general situations as in 
Theorem \ref{the-strongBMintroCn}.  In \cite{Berman-OrtegaCerda} 
it was supposed that $K$ is smooth real algebraic in $\R^m$ or 
the closure of a bounded convex open subset in $\R^m$, 
and their arguments use this hypothesis in an essential way.

Note that $\mu=|z-z_0|^M \Leb_K$ for any constant $M>0$ satisfies the hypothesis of Theorem \ref{the-strongBMintroCn}.
In general it is not possible to bound from below $B_k$ by a polynomial in $k$, see Remark \ref{re-notlowerbound}. 

By
\cite[Theorems 2 and 4]{Berman-OrtegaCerda}, if $K$ is the closure of a bounded open convex subset in $\R^n$ and $Q \equiv 0$ and $\mu$ is the restriction of the Lebesgue measure on $\R^n$ to $K$, then $k^{-n} B_n$ is bounded and bounded away from $0$ on a fixed compact subset in the interior of $K$ (the behaviour of $B_k$ at boundary points is more complicated).  On the other hand, for general $\mu$ on such $K$,  by \cite{Danka-Totik}, the upper bound for $B_k$ on $K$ can not be $O(k^n)$ in general. To be precise, it was proved there that if $K$ is a smooth Jordan curve in $\C$, and $\mu_0$ is the arc measure on $K$, and   $\mu =(z-z_0)^\alpha \mu_0$ for some constant $\alpha>0$, then $B_k(z_0) \approx k^{1+ \alpha}$ as $k \to \infty$. One can see also  \cite{Kroo-Lubinsky-boundary-ball} for a similar asymptotic in the case where $K$ is the closure of the unit ball in $\R^n$.     

We note that by \cite[Corollary 2.13]{DinhNguyen-betaensemble},  if $(K,\mu,Q)$ is as in the hypothesis of Theorem \ref{the-strongBMintroCn}), then  the triple $(K,\mu,Q)$ is $1$-Bernstein-Markov in the sense that for every constant $0<\delta \le 1$, there exists a constant $C>0$ such that 
\begin{align*} 
\sup_K |s|^2 e^{-2 kQ} \le C e^{C k^{1-\delta}} \|s\|^2_{L^2(\mu, k Q)}
\end{align*}
for every $s \in \mathcal{P}_k$. This is  much weaker than our bound. Nevertheless, \cite[Corollary 2.13]{DinhNguyen-betaensemble} is applicable to a broader class of $K$.

When $K$ is smooth (no boundary) and $Q \in \mathcal{C}^{1,\delta}(K)$ for some constant $\delta>0$ (e.g, $K$ is the unit circle in $\C$ as in a classical setting), we obtain sharp bounds which have potential applications in studying sampling or interpolation problems of multivariate polynomials on maximally totally real sets in $\C^n$. The case where $K$ is compact smooth real algebraic was considered in \cite{Berman-OrtegaCerda}. Here is our next main result. 

\begin{theorem} \label{the-bergman-smoothCn} Let $K$ be a compact $\mathcal{C}^5$ maximally totally real submanifold without boundary in $\C^n$. Let $\mu$ be a smooth volume form on $K$ regarded as a measure on $\C^n$. Let $Q \in \mathcal{C}^{1,\delta}(K)$ for some constant $\delta>0$.  Let $B_k$ be the Bergman kernel function associated to $\mu$ with weight $Q$. Then there exists a constant $C>0$ such that  
\begin{align}\label{ine-almost-optimal-Bergman}
 \sup_K B_k \le C k^{n}
\end{align}
for every $k \ge 0$. 
\end{theorem}


When $K$ is smooth compact real algebraic of dimension $n$ in $\R^m$, it was proved in \cite{Berman-OrtegaCerda} that $B_k/k^n \approx 1$.  The proof of the upper bound for $B_k$ in \cite{Berman-OrtegaCerda} relies crucially on the algebraicity of $K$. Our approach to Theorems \ref{the-strongBMintroCn} and \ref{the-bergman-smoothCn} is different and is based on constructions of analytic discs partly attached to $K$, subharmonic functions on unit discs, and fine regularity of extremal plurisubharmonic envelopes associated to $K$.    

In dimension one, we refer to \cite[Theorem 4.3]{Gustafsson-Putinar-Saff-Stylianopoulos} for a similar bound when the case where $K$ is analytic, and to \cite{Danka-Totik,Totik-transaction} and references therein for  asymptotics of $B_k$ (it behaves like $k$ at regular points, but the asymptotic of $B_k$ at singular points is more complicated). 

We are not aware of any estimates of flavor similar to Theorem \ref{the-bergman-smoothCn} in the previous literature for general  smooth maximally totally real submanifolds in higher dimensions (except the real algebraic case mentioned above in \cite{Berman-OrtegaCerda}).
We refer to \cite{Antezana-interpolation,Kroo-upperbound-convex,Prymak} for more precise bounds when $K$ is a convex subset in $\R^n$.  




Our next result is the following convergence which is a consequence of Theorem \ref{the-strongBMintroCn}.

\begin{theorem} \label{th-xapxiphiKBkCn}   
Let $K$ be a compact generic Cauchy-Riemann nondegenerate 
$\mathcal{C}^5$ piecewise-smooth submanifold in $\C^n$. 
Let $Q$ be a H\"older continuous function on $K$ and $\mu$ be as in Theorem 
\ref{the-strongBMintroCn}. Then there exists constant $C>0$ such that
$$\bigg\|\frac{1}{2k}\log \tilde{B}_k - V_{K,Q}\bigg\|_{\mathcal{C}^0(\C^n)} 
\le C\, \frac{\log k}{k}\,\cdot$$
\end{theorem}

Finally  we note that we also obtain a version of 
(Bernstein-)Markov inequality for maximally totally real 
submanifolds which might be useful elsewhere; 
see Theorem \ref{the-Bernstein-totally-real} below.\\

\noindent
\textbf{Zeros of random polynomials.} 
We give an application 
of the above results to the study of equidistribution of zeros of random polynomials. 

Let $K$ be a non-pluripolar set in $\C^n$ and let $\mu$ 
be a probability measure on $\C^n$ such that the support of $\mu$ is contained in $K$ and is non-pluripolar. 
Let $Q$ be a continuous weight on $K$.  Let $\mathcal{P}_k(K)$ 
be the space of restrictions of complex polynomials of degree at most $k$ 
in $\C^n$ to $K$. Let $d_k:= \dim \mathcal{P}_k(K)$, 
and let  $p_1,\ldots, p_{d_k}$ be an orthonormal basis of 
$\mathcal{P}_k(K)$ with respect to the $L^2(\mu,k Q)$-scalar product.  
Consider the random polynomial
\begin{align}\label{eq-randompoly}
p:= \sum_{j=1}^{d_k} \alpha_j p_j,
\end{align}
where $\alpha_j$'s are complex i.i.d. random variables. 
The study of zeros of random polynomials has a long history. 
The most classical example may be the Kac polynomial where 
$n=1$, and $p_j=z^j$.

The distribution of zeros of more general random polynomials associated to  orthonormal polynomials (as in (\ref{eq-randompoly})) was considered in \cite{Shiffman-Zelditch} by observing that  $1, z, \ldots, z^k$ are an orthonormal basis of the restriction of the space of polynomials in $\C$ to $\S^1$ with respect to the $L^2$-norm induced by the Haar measure $\mu_0$ on $\S^1$. In this setting, the necessary and sufficient conditions for the distribution of $\alpha_j$'s so that the zeros of $p$ is equidistributed almost surely or in probability with respect to the Lebesgue measure $\mu_0$ on the unit circle as $k \to \infty$ are known; see \cite{Bloom-Dauvergne,Dauvergne,Hammersley,Ibragimov-Zaporozhets}.

There are many works (in one or higher dimension) following \cite{Shiffman-Zelditch}, to cite just a few, \cite{Bloom,Bloom-Levenberg-random,Bayraktar-mass-equi,Bayraktar-zeros-indiana,Bayraktar-Coman-Mariescu}. 
In all of these works, it seems to us that the question of large deviation type estimates for the equidistribution of zeros of random polynomials has not been investigated in details. As it will be clear in our proof below, the new ingredient needed for such an estimate is an quantitative rate of convergence between $1/(2k) \log \tilde{B}_k$ and the extremal function associated to $K$. This is what we obtained in Theorem \ref{th-xapxiphiKBkCn}. To state our result we need some hypothesis on $\mu$ and the distribution of $\alpha_j$'s.

Assume now that the distribution of $\alpha_j$'s is $f \Leb_\C$, where $f$ is a nonnegative bounded Borel function on $\C$ satisfying the following mild regularities:
\begin{align}\label{ine-BLdk}
\int_{|z|> r} |f| \Leb_\C \le C/r^{2}
\end{align}   
for some constant $C>0$ independent of $r$. This condition was introduced in \cite{Bloom,Bloom-Levenberg-random}. We want to study the distribution of zeros of $p$ as $k \to \infty$. We denote by $[p=0]$ the current of integration along the zero divisor $\{p=0\}$ of $p$. Note that if $n=1$, then $[p=0]$ is the  sum of Dirac masses at zeros of $p$. 

If $(K,Q, \mu)$ is Bernstein-Markov, it was proved in \cite[Theorem 4.2]{Bloom-Levenberg-random}, that  almost surely
\begin{align}\label{converg-zero}
k^{-1}[p=0] \to \ddc \log |V_{K,Q}^*|
\end{align}
as $k \to \infty$, where the convergence is the weak one  between currents. In other words, for every smooth form $\Phi$ of degree $(2n-2)$ with compact support in $\C^n$, one has 
$$k^{-1}\int_{\{p=0\}} \Phi \to \int_{\C^n} \ddc \log |V_{K,Q}^*| \wedge \Phi$$
as $k \to \infty$.   Theorem \ref{the-BMpropertycn} above thus provides us a large class of measures for which the equidistribution of zeros of $p$ holds. 

Our goal now is to obtain a rate of convergence in (\ref{converg-zero}). To this end, it is reasonable to ask for  finer regularity on $\mu$ and of the distribution of $\alpha_j$. We don't try to make the most optimal condition.  Here is our hypothesis:

(H1) $|f(z)| \le |z|^{-3}$ for $|z|$ sufficiently large.  

(H2) let $K$ be a non-degenerate $\mathcal{C}^5$ piecewise-smooth generic Cauchy-Riemann submanifold of $\C^n$, and $Q$ be a H\"older continuous function on $K$. Let $\mu= \rho \Leb_K$, where $\rho^{-\lambda} \in L^1(\Leb_K)$ for some constant $\lambda>0$.

The condition (H1) ensures that (\ref{ine-BLdk}) holds, and the joint-distribution of $\alpha_1,\ldots, \alpha_{d_k}$ is dominated by the Fubini-Study volume form on $\C^{d_k}$ (by definition the Fubini-Study volume form is equal to $\omega_{FS}^{d_k}$, where $\omega_{FS}$ is the Fubini-Study form on $\P^{d_k} \supset \C^{d_k}$). Clearly the Gaussian random variables satisfy this condition.

The condition (H2) is a natural generalization of the classical setting with Kac polynomials where $K$ is the unit circle in $\C$. Indeed in \cite{Shiffman-Zelditch} the authors considered the setting in which $\mu$ is the surface area on a closed analytic curve in $\C$ bounding  a simply connected domain $\Omega$ in $\C$ or $\mu$ is the restriction of the Lebesgue measure on $\C$ to $\Omega$. This setting is relevant to the theory of random matrix theory as already pointed out in \cite{Shiffman-Zelditch}. We refer to \cite{Bloom-Dauvergne,Pritsker-Ramachandran,Pritsker-Ramachandran2} for partial generalizations (without quantitative estimates) to domains with smooth boundary in $\C$. We would like to mention also that in some cases,  certain large deviation type estimates for random polynomials in dimension one were known; see \cite[Theorem 10]{Goetze-Jalowy} for polynomial error terms, and \cite[Theorem 1.1]{Dinh-random}, and \cite[Theorem 3.10]{DV_random} for exponential error terms. 

To our best knowledge there has been no quantitative generalization of results in \cite{Shiffman-Zelditch} to higher dimension. It was commented in the last paper that their method seems to have no simple generalization to the case of higher dimension.

We now recall distances on the space of currents. For every constant $\beta  \ge 0$, and $T,S$ closed positive currents of bi-degree $(m,m)$ on  the complex projective space $\P^n$, define
$$\dist_{-\beta}(T,S):= \sup_{\Phi:\, \|\Phi\|_{\mathcal{C}^{[\beta], \beta- [\beta]}} \le 1} |\langle T-S, \Phi \rangle|,$$  
where $[\beta]$ denotes the greatest integer less than or equal to $\beta$, and $\Phi$ is a smooth form of degree $(2n-m)$ on $\P^n$. 

It is a standard fact that the distance $\dist_{-\beta}$ for $\beta>0$ induces the weak topology on the space of closed positive currents (see for example \cite[Proposition 2.1.4]{DinhSibony_Pk_superpotential}). We have the following interpolation inequality: for $0< \beta_1 \le \beta_2$, there is a constant $c_{\beta_1,\beta_2}$ such that
\begin{align} \label{ine-interpolar}
\dist_{-\beta_2}  \le \dist_{-\beta_1} \le c_{\beta_1,\beta_2} [\dist_{-\beta_2}]^{\beta_1/\beta_2};
\end{align} 
see \cite[Lemma 2.1.2]{DinhSibony_Pk_superpotential} or \cite{Lunardi-book-interpolation,Triebel}.

Note that the currents $[p=0]$ and $\ddc V_{K,Q}$ extends trivially through the hyperplane at infinity $\P^n \backslash \C^n$ to be closed positive currents of bi-degree $(1,1)$ on $\P^n$ (this is due to the correspondence (\ref{corres-Leleongclass}) above).  Hence one can consider $\dist_{-\beta}$ between $k^{-1}[p=0]$ and $\ddc V_{K,Q}$ as closed positive currents on $\P^n$.

\begin{theorem}\label{the-zeros} (A large deviation type estimate) Let $M \ge 1$ be a constant. Assume that (H1) and (H2) are satisfied. Then there exists a constant $C_M>0$ so that  
\begin{align} \label{ine-largedevirationSZ}
\cali{P}_k\bigg \{(\alpha_1, \ldots, \alpha_{d_k}) \in \C^{d_k}: \dist_{-2}\big(k^{-1}[p=0], \ddc V_{K,Q}\big) \ge \frac{C_M\log k}{k} \bigg\} \le C_M k^{-M},
\end{align}
for every $k$, where $\cali{P}_k$ denotes the joint-distribution of $\alpha_1,\ldots, \alpha_{d_k}$. 
\end{theorem}

By (\ref{ine-interpolar}), one obtains similar estimates for $\dist_{-\beta}$ with $0< \beta \le 2$ as in Theorem \ref{the-zeros}.  We don't know if the right-hand side of (\ref{ine-largedevirationSZ}) can be improved.  We state now a direct consequence of Theorem \ref{the-zeros} which gives a higher dimensional generalization of \cite[Theorems 1 and 2]{Shiffman-Zelditch} (except that we only obtain the error term $O(\frac{\log k}{k})$ instead of $O(k^{-1})$); see also Theorem \ref{the-zerosgiaoL} below. Denote by $\E_k(k^{-1}[p=0])$ the expectation of the  random normalized currents of zeros $k^{-1}[p=0]$. 

\begin{corollary}\label{cor-expectedzeros}  
Assume that (H1) and (H2) are satisfied. Then we have 
\begin{align} \label{eq-corexpeczeros}
\E_k\big(k^{-1}[p=0]\big) = 
\ddc V_{K,Q}+ O\left(\frac{\log k}{k}\right),
\end{align}
where $O(\frac{\log k}{k})$ denotes a current $S_k$ 
of order $0$ in $\C^n$ such that for every smooth form $\Phi$ 
of degree $(2n-m)$ with compact support in $\C^n$ 
such that $\|\Phi\|_{\mathcal{C}^2} \le 1$, we have 
$$| \langle S_k, \Phi \rangle | \le C \frac{\log k}{k}\,,$$
for some constant $C$ independent of $k, \Phi$. 
\end{corollary}

Even when the $\alpha_j$'s are Gaussian variables, 
we underline that the decay obtained in \cite{Shiffman-Zelditch} 
is only $O(k^{-1})$, this error term is  optimal in dimension 1 
(one can see it by a careful examination of computations in
\cite[Proposition 3.3]{Shiffman-Zelditch}). 

Since zeros of random polynomials in higher dimension form no 
longer a discrete set, one might be somehow not at ease to speak 
of questions like correlation of zeros.  To remedy this problem, 
one can reformulate the equidistribution of zeros of random polynomials 
in the following way. Let $L$ be a  complex line in $\C^n$ or an (complex) 
algebraic curve  in $\C^n$.  Sine generic polynomials 
intersect transversely $L$, almost surely the number of intersection 
points (without counting multiplicities) of the random hypersurface 
$\{p=0\}$ and $L$ is exactly $k\deg L$ by Bezout's theorem. Define 
$$\mu_{k,L}:= \frac{1}{k \deg L}\sum_{j=1}^{k \deg L} \delta_{z_j},$$
where $z_1, \ldots, z_{k \deg L}$ are zeros of $p$ on $L$. 
Let $[L]$ be the current of integration along $L$. 
Since $V_{K,Q}$ is bounded, the product 
$$\mu_L:= \frac{1}{\deg L}\ddc V_{K,Q} \wedge [L]$$
 is  a well-defined measure supported on $L$ 
 (it is simply $\ddc (V_{K,Q}|_{L})$ if $L$ is smooth). 

\begin{theorem}\label{the-zerosgiaoL} Let $M \ge 1$ be a constant.  Assume that (H1) and (H2) 
are satisfied. Then there exists a constant $C_M>0$ so that  
$$\cali{P}_k\bigg \{(\alpha_1, \ldots, \alpha_{d_k}) \in \C^{d_k}:
\dist_{-2}\big(\mu_{k,L}, \mu_L \big) \ge \frac{C_M\log k}{k}  \bigg\} 
\le C_M k^{-M},$$
for every $k$. 
In particular, the measure $\mu_{k,L}$ converges weakly to 
$\mu_L$ as $k \to \infty$. 
\end{theorem}

Now since zeros of $p$ on $L$ is discrete and is equidistributed 
as $k \to \infty$, one can ask as in \cite{Shiffman-Zelditch} 
how zeros of $p$ on $L$ (if scaled appropriately) are correlated. 
Nevertheless such questions seem to be still out of reach in 
the higher dimensional setting. Finally we note that one can 
even consider $L$ to be a transcendental curve in $\C^n$. 
In this case  generic polynomials $p$  still intersect $L$ 
transversely asymptotically (see \cite{VietTuan}); 
the question of equidistribution is however more involving. \\


\noindent
\textbf{Acknowledgement.} We thank Norman Levenberg for many fruitful discussions.

\section{Bergman kernel functions associated to a line bundle}

The results mentioned in the Introduction have their direct generalizations 
in the context of complex geometry where $\C^n$ is replaced 
by a compact K\"ahler manifold. Working in such a generality 
will make the presentation more clear and enlarge the range of applicability of the theory.   
We will now describe the setting. 
    
Let $X$ be a projective manifold of dimension $n$.  
Let $(L,h_0)$ be an ample line bundle equipped with 
a Hermitian metric $h_0$ whose Chern form $\omega$ is positive. 
Let $K$ be a compact non-pluripolar subset in $X$.  
Let $\mu$ be a probability measure on $X$ such that the support 
of $\mu$ is non-pluripolar and is contained in $K$. 
Let $h$ be a Hermitian metric on $L|_{K}$ such that 
$h= e^{-2\phi} h_0$, where $\phi$ is a continuous function on $K$.
For $s_1,s_2 \in H^0(X,L)$, we define
$$\langle s_1, s_2 \rangle:= \int_X \langle s_1, s_2 \rangle_h d \mu.$$
Since $\supp \mu$ is non-pluripolar, the last scalar product defines a norm 
called $L^2(\mu,h)$-norm on $H^0(X,L)$.  Let $k \in \N$. 
We obtain induced Hermitian metric $h^k$ on $L^k$ and a similar norm 
$L^2(\mu, h^k)$ on $H^0(X,L^k)$.  Put $d_k:= \dim H^0(X,L^k)$. 
Let $\{s_1, \ldots, s_{d_k}\}$ be an orthonormal basis of $H^0(X,L^k)$ 
with respect to $L^2(\mu, h^k)$. 
The \emph{Bergman kernel function of order $k$ associated to $(L,h,\mu)$} is 
$$B_k(x):= \sum_{j=1}^{d_k} |s_j(x)|_{h^k}^2=
\sup\big\{ |s(x)|_{h^k}^2: \quad s \in H^0(X,L^k), 
\quad \|s\|_{L^2(\mu,h^k)}=1\big\}$$
for $x \in K$. 

When $\mu$ is a volume form on $X$ and $h=h_0$, 
the Bergman kernel function is an object of great importance 
in complex geometry, for example see \cite{Marinescu-Ma} 
for a comprehensive study.

The setting considered in Introduction corresponds to the case where 
$X= \P^n$ and $(L,h_0)=(\mathcal{O}(1),h_{FS})$ is the 
hyperplane line bundle on $\P^n$ 
endowed with the Fubini-Study metric. We consider $\C^n$ 
as an open subset in $\P^n$ and the weight $Q$ corresponds to 
$\phi - \frac{1}{2}\log (1+ |z|^2)$. Recall that there is a natural 
identification between $\mathcal{L}(\C^n)$ and the set of 
$\omega_{FS}$-psh functions on $\P^n$ (where $\omega_{FS}$ 
denotes the Fubini-Study form on $\P^n$) given by 
\begin{align} \label{corres-Leleongclass}
u \, \longleftrightarrow \, u- \frac{1}{2}\log(1+ |z|^2)
\end{align}
for $u \in \mathcal{L}(\C^n)$. 

Another well-known example is the case where $K$ is the unit sphere in 
$\R^n$ (here $n\ge 2$; see, e.g, \cite{Marzo-OrtegaCerda}) 
and $X$ is the complexification of $K$, \emph{i.e,} 
$K= \S^{n-1} \subset \R^n$ which is considered as usual 
a compact subset of $X:= \{z_0^2+ z_1^2+ \cdots+ z_n^2=1\} \subset \P^n$. The line bundle $L$ on 
$X$ is the restriction of the hyperplane bundle $\mathcal{O}(1)\to\P^n$ to $X$. 
We remark that in this case $H^0(X, L^k)$ is equal to the restriction 
of the space of $H^0(\P^n, \mathcal{O}(k))$ to $X$. 
Hence  the restriction of $H^0(X,L^k)$ to $K$ is that of the space of 
complex polynomials in $\C^n$ to $K$.  
To see this, notice that $X$ is a smooth hypersurface in $\P^n$. 
Consider the standard exact sequence of sheaves:
$$0 \to \mathcal{O}(k- \deg X)\to\mathcal{O}(k)\to 
\mathcal{O}(k)|_X \rightarrow 0,$$
where the second arrow is the multiplication by a section of 
$\mathcal{O}(\deg X)$ whose zero divisor is equal to $X$.  
We thus obtain a long exact sequence of cohomology
spaces: 
\begin{multline*}
0 \to H^0(\P^n,\mathcal{O}(k - \deg X)) \to 
H^0(\P^n,\mathcal{O}(1)) \to  H^0(\P^n,\mathcal{O}(k)|_X)\\
 \to H^1(\P^n,\mathcal{O}( k- \deg X)) \to \cdots
 \end{multline*}
In the last sequence, $H^0(\P^n,\mathcal{O}(k)|_X)$
is exactly equal to $H^0(X,\mathcal{O}(k)|_X)$, and 
$H^1(\P^n,\mathcal{O}(k- \deg X))= 0$
by Kodaira-Nakano vanishing theorem; see \cite[p.\ 156]{Griffiths-Harris}.
As above, the weight $Q$ on $K$ in the spherical model corresponds to 
$\phi- \frac{1}{2}\log (1+ |z|^2)|_X$ in the setting $(K,X, \mathcal{O}(1)|_X)$.


The measure $\mu$ is said to be a \emph{Bernstein-Markov measure} 
(with respect to $(K,\phi,L)$) if  for every $\epsilon>0$ 
there exists $C=C(\epsilon)>0$ such that  
\begin{align}\label{ine-def-BMintro}
\sup_K |s|_{h^k}^2 \le C e^{\epsilon k} \|s\|^2_{L^2(\mu, h^k)}
\end{align}
for every $s \in H^0(X, L^k)$. In other words, the Bergman kernel function 
of order $k$ grows at most subexponentially, \emph{i.e,}  
$\sup_K B_k =O(e^{\epsilon k})$ as $k \to \infty$ for every $\epsilon>0$. 
Theorem \ref{the-BMpropertycn} is a particular case of the following result.

\begin{theorem} \label{the-BMproperty}   
Let $K$ be a compact nondegenerate $\mathcal{C}^5$ piecewise-smooth 
Cauchy-Riemann generic   submanifold  of $X$. 
Then for  every continuous function $\phi$ on $K$, if $\mu$ 
is a finite measure whose support is equal to $K$ such that there 
exist constants $\tau>0, r_0>0$ satisfying 
$\mu(\B(z,r) \cap K) \ge r^\tau$ 
for every $z \in K$, and every $r \le r_0$  (where $\B(z,r)$ 
denotes the ball of radius $r$ centered at $z$ induced by 
a fixed smooth Riemannian metric on $X$), 
then $\mu$ is a Bernstein-Markov measure with respect to $(K,\phi,L)$.
\end{theorem}

Let 
\begin{equation}\label{fik}
\phi_K:= \sup\{\psi \in \PSH(X, \omega): 
\, \psi \le \phi \: \text{ on } K\}.
\end{equation}
Since $K$ is non-pluripolar, the function $\phi_K^*$ is a bounded 
$\omega$-psh function. If $\phi_K= \phi^*_K$, then we say 
$(K,\phi)$ is \emph{regular}. A stronger notion is the following: 
we say that $K$ is \emph{locally regular} if for every $z \in K$ 
there is an open neighborhood $U$ of $z$ such that for every 
increasing sequence of psh functions $(u_j)_j$ on $U$ with 
$u_j \le 0$ on $K \cap U$, then 
$$(\sup_j u_j)^* \le 0$$
on $K \cap U$.  The following result answers the question 
raised in \cite[Remark 1.8]{BoucksomBermanWitt}.

\begin{theorem} \label{the-BMpropertylocallyregu}
Every compact nondegenerate $\mathcal{C}^5$ piecewise-smooth 
Cauchy-Riemann generic submanifold of $X$ is locally regular. 
\end{theorem}

Note that Theorem \ref{the-BMpropertylocallyreguCn} is a 
direct consequence of the above result. 
It was shown in \cite[Corollary 1.7]{BoucksomBermanWitt} 
that $K$ is locally regular if $K$ is smooth real analytic.
Theorem \ref{the-BMproperty} is actually a direct consequence of 
Theorem \ref{the-BMpropertylocallyregu} and the 
criterion \cite[Proposition 3.4]{Bloom-Levenberg-PW} 
giving a sufficient condition for measures being Bernstein-Markov.

The Monge-Amp\`ere current $(\ddc \phi_K^* + \omega)^n$ is called 
\emph{the equilibrium measure} associated to $(K,\phi)$. 
It is well-known that the last measure is supported on $K$.
By \cite[Theorem B]{BoucksomBermanWitt} one has 
\begin{align}\label{conver-weak-bergmankernelintro}
d_k^{-1} B_k \mu \to (\ddc \phi_K^*+ \omega)^n\,,\quad k\to\infty,
\end{align}
provided that $\mu$ is a Bernstein-Markov measure associated to $(K,\phi,L)$.
The last property suggests that the Bergman kernel function $B_k$ 
cannot behave too wildly at infinity.  

\begin{theorem} \label{the-strongBMintro}  Let $K$ be a compact nondegenerate $\mathcal{C}^5$ piecewise-smooth Cauchy-Riemann generic   submanifold of $X$. Let $n_K$ be the dimension of $K$. Let   $\phi$ be a H\"older continuous function of H\"older exponent $\alpha \in (0,1)$ on $K$, let $\Leb_K$ be a smooth volume form on $K$, and  $\mu = \rho \Leb_K$, where $\rho \ge 0$ and $\rho^{-\lambda} \in L^1(\Leb_K)$ for some  constant $\lambda>0$. Then, there exists a constant $C>0$ such that 
$$\sup_K B_k \le C k^{2n_K (\lambda+1)/(\alpha\lambda)}$$
for every $k$. 
\end{theorem}


Note that by  the  proof of \cite[Theorem 1.3]{Dinh-Ma-Marinescu} 
or \cite[Theorem 3.6]{DMN}, for every H\"older continuous function 
$\phi_1$ on $X$, and  $\mu_1:= \omega^n$,  the Bergman kernel function 
of order $k$ associated to  $(X, \mu_1, \phi_1)$ grows at most polynomially 
on $K$ as $k \to \infty$; see also \cite[Theorem 3.1]{BoucksomBerman} 
for the case where $\phi_1$ is smooth.

\begin{theorem} \label{the-bergman-smooth} Assume that the following two conditions hold:

(i) $K$ is maximally totally real and  has no singularity (i.e, $K$ is smooth and without boundary),

(ii) $\phi \in \mathcal{C}^{1,\delta}(K)$ for some constant $\delta>0$.

Then there exists a constant $C>0$ such that  for every $k$, 
the following holds 
$$B_k(x) \le C k^{n},$$
for every $x \in K$.
\end{theorem}

Consider the case when $X= \P^n$, $L:=\mathcal{O}(1)$, $h_0=h_{FS}$ 
is be the Fubini-Study metric on $\mathcal{O}(1)$, 
and $K$ is a smooth maximally totally real compact submanifold in 
$ \C^n \subset \P^n$, and $h$ is the trivial line bundle on $K$. 
It is clear that $\phi:= Q -\frac{1}{2}\log (1+|z|^2)$ is
in $\mathcal{C}^{1,\delta}(K)$ if $Q$ is so. 
In this case the hypothesis of Theorem  \ref{the-bergman-smooth} are fulfilled.  
Thus Theorem \ref{the-bergman-smooth}
implies Theorem \ref{the-bergman-smoothCn}.


As a consequence of Theorem \ref{the-strongBMintro}, we obtain the following estimate generalizing Theorem \ref{th-xapxiphiKBkCn} in Introduction.

\begin{theorem} \label{the-convergencebergmankernel} 
Let $K$ be a compact  nondegenerate $\mathcal{C}^5$ piecewise-smooth 
generic submanifold $K$ of $X$.   Let $\phi$ be a H\"older 
continuous function on $K$. Let $\mu$ be a smooth volume form on $K$. 
Then we have 
$$\left\| \frac{1}{2k} \log \tilde{B}_k - \phi_K \right\|_{\mathcal{C}^0(X)}=
O\left(\frac{\log k}{k}\right),$$
as $k \to \infty$, where 
$$\tilde{B}_k:= e^{2 k \phi} B_k=
\sup\big\{ |s(x)|_{h_0^k}^2: \: s \in H^0(X,L^k), \: \|s\|_{L^2(\mu,h^k)}=1\big\}.$$
\end{theorem}

\section{Bernstein-Markov property for totally real submanifolds}

In the first part of this section we prove Theorem \ref{the-BMpropertylocallyregu}, and hence Theorem \ref{the-BMproperty} as commented in the paragraph after Theorem \ref{the-BMpropertylocallyregu}. In the second part of the section, assuming Theorem \ref{the-strongBMintro}, we prove  Theorem \ref{the-convergencebergmankernel}. 

\subsection{Local regularity}

Let $X$ be a compact K\"ahler manifold of dimension $n$ with a K\"ahler form $\omega$.  Let $K$ be a compact non-pluripolar subset on $X$ and $\phi$ be a continuous function on $K$. Recall 
 $$\phi_{K}:= \sup\big\{ \psi: \quad  \psi  \quad \omega\text{-psh},\quad  \psi \le \phi   \quad \text{on}\quad  K  \big \}.$$
By non-pluripolarity of $K$, we have $\phi_K < \infty$.  Hence $\phi^*_K$ is a bounded $\omega$-psh function on $X$. 

When $K= X$ and $\phi \in \mathcal{C}^{1,1}$, it was proved in \cite{Berman-C11regula,Tosatti-envelop,Chu-Zhou-optimal-envelope} that $\phi_K \in \mathcal{C}^{1,1}$. In general, the best regularity for $\phi_K$ is H\"older one, see Theorem \ref{the1CalphaClapharegu} below and comment following it.
 One can check that if $K$ is locally regular, then  $(K,\phi)$ is regular for every $\phi$.

Let $\D$ be the open unit disc in $\C.$ \emph{An analytic disc} $f$ in $X$ is a holomorphic mapping from $\D$ to $X$ which is continuous up to the boundary $\partial \D$ of $\D.$ For an interval $I \subset \partial \D,$ $f$ \emph{is said to be $I$-attached to a subset $E \subset X$} if $f(I) \subset E.$ Fix a Riemannian metric on $X$ and denote by $\dist(\cdot, \cdot)$ the distance induced by it.  For $x \in X$ and $r \in \R^+$, let $\B(x,r)$ be the ball of radius $r$ centered at $x$ with respect to the fixed metric.  Here is the crucial property for us showing the existence of well-behaved analytic discs partly attached to a generic Cauchy-Riemann submanifold. 
 
\begin{proposition}\label{pro_familydiscK} (\cite[Proposition 2.5]{Vu_feketepoint})   Let $K$ be a compact generic nondegenerate $\mathcal{C}^5$ piecewise-smooth submanifold of $X$. Then, there are  positive constants $c_0, r_0$ and $\theta_0 \in (0, \pi/2)$  such that for any $a_0 \in K$ and any $a\in \B(a_0,r_0)\backslash \{a_0\},$ there exist  a $\mathcal{C}^2$ analytic disc $f: \overline{\D} \rightarrow X$ such that $f$ is $[e^{-i\theta_0}, e^{i\theta_0}]$-attached to $K,$ $\dist(f(1),a_0)\le c_0 \delta$ with $\delta=\dist(a,a_0),$ $\|f\|_{\mathcal{C}^2} \le c_0$, and there is  $z^* \in \D$ so that $|1-z^*| \le \sqrt{c_0 \delta}$ and $f(z^*)=a.$ Moreover if $a_0$ is in a fixed compact subset $K'$ of the regular part of $K$, then we have  $|1-z^*| \le c_0 \delta$.
\end{proposition}

It was stated that $f \in \mathcal{C}^1(\overline \D)$ instead of 
$f \in \mathcal{C}^2(\overline \D)$ in \cite[Proposition 2.5]{Vu_feketepoint}. 
But the latter regularity is indeed clear from the construction in the proof of 
\cite[Proposition 2.5]{Vu_feketepoint}.  
Note that the compactness of $X$ is not necessary in the above result. 
In particular, if $K \Subset U$ for some open subset $U$ of $X$, 
then the analytic disc $f$ can be chosen to lie entirely in $U$. 
Here is a slight improvement of Proposition \ref{pro_familydiscK}.

\begin{proposition} \label{pro-hopcuaK} Assume that one of the 
following assumptions hold:

(1) $K$ is a compact generic $\mathcal{C}^5$ smooth submanifold
with $\mathcal{C}^5$ smooth boundary that is also generic,

(2) $K$ is a union of a finite number of compact sets as in (1). 

Then there are  positive constants $c_0, r_0$ and 
$\theta_0 \in (0, \pi/2)$ such that for any $a_0 \in K$ 
and any $a\in \B(a_0,r_0)\backslash \{a_0\},$ there exist
a $\mathcal{C}^2$ analytic disc $f: \overline{\D} \rightarrow X$ 
such that $f$ is $[e^{-i\theta_0}, e^{i\theta_0}]$-attached to $K,$ 
$\dist(f(1),a_0)\le c_0 \delta$ with $\delta=\dist(a,a_0),$ 
$\|f\|_{\mathcal{C}^2} \le c_0$, and there is  $z^*\in \D$
so that  $|1-z^*| \le c_0 \delta$.
\end{proposition}

\proof 
If $K$ fulfills one of the conditions (1) or (2)
then 
$K$ can be  covered by a finite number of sets $K_j$ 
such that for every $j$ there exist an open subset 
$U_j$ in $X$ and 
a smooth family $(K_{js})_{s \in S_j}$ 
of $\mathcal{C}^5$ smooth generic CR submanifolds 
$K_{js}$ in $U_j$ such that $K_{js}$ is $\mathcal{C}^5$ 
smooth without boundary in $U_j$, for every $s$,
and satisfies $K_j= \cup_{s \in S_j} K_{js}$. 
Now the desired assertion follows directly from Proposition \ref{pro_familydiscK} 
applied to each $K_{js}$ and points in $U_j$ correspondingly. We note that the constants $c_0,r_0,\theta_0$ can be chosen independent of $s \in S_j$ because as shown in the proof of \cite[Proposition 2.5]{Vu_feketepoint}, they depend only on bounds on $\mathcal{C}^3$-norm of diffeomorphisms defining local charts in $K_{js}$ (see \cite[Lemma 4.1]{Vu_feketepoint}), these bounds are independent of $s \in S_j$ for the family $(K_{js})_{s \in S_j}$ is smooth.   
\endproof
 
Examples for compact $K$ satisfying the hypothesis of 
Proposition \ref{pro-hopcuaK} include a collection of finite number of 
smooth Jordan arcs in $\C$ regardless of their configuration or the 
closure of an open subset with smooth boundary in $X$.  
 
\begin{lemma}\label{le_passagetogeneralcase3} Let $\theta_0 \in (0, \pi/6), \beta \in (0,1)$ and let $c>0$ be a constant. Let $\psi$ be a subharmonic function on $\D$.  Assume that 
\begin{align}\label{ine_giasutrenpsi}
\limsup_{z \in \D \to e^{i \theta}}\psi(z) \le c|\theta|^{\beta} \quad \text{for} \quad \theta \in (-\theta_0,\theta_0) \quad \text{and} \quad  \sup_\D \psi \le c.
\end{align}
Then, there exists a constant $C$ depending only on $(\theta_0,\beta,c)$ so that for any $z \in \D,$  we have
\begin{align} \label{ine_passagetogeneralcase4cu}
\psi(z) \le  C |1-z|^{\beta}.
\end{align}  
Moreover if  $\limsup_{z \in \D \to e^{i \theta}}\psi(z) \le g(e^{\theta})$ for some function $g \in \mathcal{C}^{1,\delta}$ on $[e^{-i \theta_0}, e^{i \theta_0}]$ and for some $\delta>0$ so that $g(1)=0$, then 
\begin{align} \label{ine_passagetogeneralcase4}
\psi(z) \le  C |1-z|,
\end{align} 
for some constant $C$ independent of $z \in \D$.
\end{lemma}

\proof The desired inequality (\ref{ine_passagetogeneralcase4cu}) is essentially contained in \cite[Lemma 2.6]{Vu_feketepoint}.    The hypothesis of continuity up to boundary of $\psi$ in the last lemma is superfluous and the proof there still works in our current setting. Note that the proof of \cite[Lemma 2.6]{Vu_feketepoint} does not work for $\beta=1$ because the harmonic extension of a Lipschitz function on $\partial \D$ is not necessarily Lipschitz on $\overline \D$. However since the harmonic extension of a $\mathcal{C}^{1,\delta}$ function on $\partial \D$ to $\D$ is  also $\mathcal{C}^{1,\delta}$ on $\overline \D$ (see, e.g, \cite[Page 41]{Gakhov}),  we obtain (\ref{ine_passagetogeneralcase4}).
\endproof

\begin{proof}[End of the proof of  Theorem \ref{the-BMpropertylocallyregu}] Let $a_0\in K$ and a small ball $\B$ of $X$ around $a_0$. Consider an increasing  sequence $(u_j)_j$ of psh functions bounded uniformly from above on $\B$ such that $u_j \le 0$ on $K\cap \B$. We need to check that $(\sup_j u_j)^* \le 0$ on $K \cap \B$. Now, we will essentially follow arguments from the proof of \cite[Theorem 2.3]{Vu_feketepoint}. Let $\B'$ be a relatively compact subset of $\B$ containing $a_0$. We will check that there exists a constant $C>0$ such that for every $a \in \B'$, we have
\begin{align} \label{ine-danhgiau_j}
u_j(a) \le C\dist(a, K)^{1/5}.   
\end{align}
The desired assertion is deduced from the last inequality by taking $\dist(a,K) \to 0$. It remains to check  (\ref{ine-danhgiau_j}). 

Let $a'_0$ be a point in $K$ such that $\dist(a,a'_0)= \dist(a,K)$.  Put $\delta:= \dist(a,K)= \dist(a,a'_0)$. By  Proposition \ref{pro_familydiscK}, there exists an analytic disc $f: \D \to \B$ continuous up to boundary and $z_a \in \D$ with $|z_a - 1| \le C \delta^{1/2}$ such that $f(z_a)= a$ and $\dist\big(f(1),a'_0\big) \le C \delta$, and $f([e^{-i \theta_0}, e^{i\theta_0}]) \subset K$, for some constants $C$ and $\theta_0$ independent of $a$.

 Put $v_j:= u_j \circ f$. Since $u_j \le 0$ on $\B \cap K$ and $f([e^{-i \theta_0}, e^{i\theta_0}]) \subset K$, we get $v_j(e^{i \theta}) \le 0$ for $\theta \in [-\theta_0, \theta_0]$. Moreover since $u_j$ is uniformly bounded from above, there is a constant $M$ such that $v_j \le M$ for every $j$. This allows us to apply Lemma \ref{le_passagetogeneralcase3} for $\beta=1-\epsilon$ (for some constant $\epsilon>0$ small) and $c$ big enough. We infer that $v_j(z) \lesssim |1-z|^{1/2}$. Substituting $z= z_a$ in the last inequality gives
$$u_j(a)= u_j\big(f(z_a)\big)= v_j(z_a) \lesssim |1- z_a|^{(1- \epsilon)/2} \lesssim \delta^{(1- \epsilon)/4}.$$ 
Hence, (\ref{ine-danhgiau_j}) follows by choosing $\epsilon:=1/5$. The proof is finished.        
\end{proof}

Let $h_0, h, \phi$ be as in the previous section. Recall that the Chern form of $h_0$ is equal to $\omega$.  Define 
$$\phi_{K,k}:=\sup\big\{ k^{-1} \log |\sigma|_{h_0}: \quad  \sigma \in H^0(X, L^k), \,  \sup_K (|\sigma|_{h_0^k} e^{-k\phi}) \le 1\big\}.$$
Clearly $\phi_{K,k} \le \phi_K$. We recall the following well-known fact.

\begin{lemma} \label{le-pointwiseconvergence} The sequence $(\phi_{K,k})_k$ increases pointwise to $\phi_K$ as $k \to \infty$.
\end{lemma}

As a direct consequence of the above lemma, we see that  $\phi_K$ is lower semi-continuous. 

\proof  Since we couldn't find a proper reference, we present detailed arguments here.
We just need to use Demailly's analytic approximation. 
Since $\phi_K$ is bounded, without loss of generality, 
we can assume that $\phi_K <0$. Clearly, $\phi_{K,k} \le \phi_K$. 
Fix $a \in X$. Let $\delta$ be a positive constant.
Let $\psi$ be a negative $\omega$-psh function with 
$\psi \le \phi$ on $K$ such that $\psi(a) \ge \phi_K(a)-\delta$.
Let $\epsilon \in (0,1)$. Observe 
$\ddc (1-\epsilon)\psi+\omega \ge \epsilon \omega$. 
This allows us to apply \cite[Theorem 14.21]{Demailly_analyticmethod} 
to $\psi$.
Let $(\sigma_j)_j$ is an orthonormal basis of  $H^0(X, L^k)$ 
with respect to $L^2$-norm generated by the Hermitian metric 
$h_{\epsilon,\psi,k}:= e^{-k(1-\epsilon)\psi} h_0^k$ and $\omega^n$.
Set 
$$\psi_{\epsilon,k}:=\frac{1}{k}\log\sum_{j=1}^{d_k}
|\sigma_j|_{h_0^k}$$
Then
$$\psi_{\epsilon,k} \ge (1-\epsilon) \psi$$ and 
$\psi_k$ converges pointwise to $(1-\epsilon)\psi$ as $k \to \infty$. Note that
$$\psi_{\epsilon,k}=\sup\left\{\frac{1}{k}\log |\sigma|_{h_{0}^k}:
\sigma \in H^0(X, L^k): \|\sigma\|_{L^2(\omega^n, h_{\epsilon,\psi,k})}=1\right\} .$$ 
Let $(\psi^{(N)})_N$ be a sequence of continuous functions decreasing to $\psi$ as $N\to \infty$. Using Hartog's lemma applied to $(\psi_{1/N, k})_{k \in \N}$, 
we see that  there is a sequence $(k_N)_N\subset \N$ increasing to 
$\infty$ such that
$$(1- 1/N) \psi \le \psi_{\frac{1}{N}, k_N} \le (1-1/N) \psi^{(N)}+ 1/N.$$
Consequently, $\psi_{\frac{1}{N},k_N}$ converges pointwise to $\psi$ as $N \to \infty$. 


 Recall that $\psi \le \phi$ on $K$.  
 By Hartog's lemma again and the continuity of $\phi$,  
 for every constant $\delta'>0$ and $N$ large enough, we have  
\begin{align}\label{ine-psi'kxapxiphi}
\psi_{\frac{1}{N},k_N} \le \phi+ \delta'
\end{align}
on $K$.  It follows that 
\begin{align}\label{ine-psi'kxapxiphi2}
k_N^{-1}\log|\sigma|_{h_0^{k_N}} \le  \phi +\delta'
\end{align}
for every $\sigma \in H^0(X, L^{k_N})$ with 
$\|\sigma\|_{L^2(\omega^n,h_{\frac{1}{N},\psi, k_N})}=1$. 
 We deduce that 
$$\phi_{K,k_N} \ge k_N^{-1}\log |\sigma|_{h_0^{k_N}} -\delta'$$
for such $\sigma$. In other words, $\phi_{K,k_N} \ge 
\psi_{\frac{1}{N}, k_N}-\delta'$ for $N$ big enough. 
Letting $N \to \infty$ gives
$$\liminf_{N \to \infty} \phi_{K,k_N}(a) \ge 
\lim_{N \to \infty}\psi_{\frac{1}{N}, k_N}(a)-\delta'= 
\psi(a)- \delta' \ge \phi_K(a)- \delta'-\delta.$$
Letting $\delta, \delta'$ tend to $0$ yields that 
$\liminf_{N \to \infty} \phi_{K,k_N}(a)= \phi_{K}(a)$. 
Hence $\phi_{K,k_N} \to \phi_K$ as $N \to \infty$. 
We have actually shown that for every sequence 
$(k'_N)_N \subset \N$ converging to $\infty$, 
there is a subsequence $(k_N)_N$ such that $\phi_{K,k_N}$ 
converges to $\phi_K$. Thus the desired assertion follows. 
\endproof

Using arguments from  \cite[Lemma 3.2]{Bloom-Shiffman} and Lemma \ref{le-pointwiseconvergence} gives the following.

\begin{lemma} \label{le-hoitudeuenvelop} Assume that $(K,\phi)$ is regular. Then $\phi_{K,k}$ converges uniformly to $\phi_K$ as $k \to \infty$. 
\end{lemma}

\proof For readers' convenience, we briefly recall the proof here. Since $\phi_K= \phi_K^*$, we see that $\phi_K$ is upper semi-continuous. This combined with the fact that $\phi_K$ is already lower semi-continuous gives that $\phi_K$ is continuous. By Lemma \ref{le-pointwiseconvergence}, we have the pointwise convergence of $\phi_{K,k}$ to $\phi_K$. Using the envelop defining $\phi_{K}$, observe next that 
\begin{align}\label{ine-subaditivephi_k}
k \phi_{K,k}+ m \phi_{K,m} \le (k+m)\phi_{K, k+m}
\end{align}
for every $k,m$. 

We fix a Riemannian metric $d$ on $X$. Let $\epsilon>0$.   
Since $X$ is compact, $\phi_K$ is uniformly continuous on $X$. 
Hence there exists a constant $\delta>0$ such that 
$d\big(\phi_K(x), \phi_K(y)\big) \le \epsilon$ if 
$d(x,y) \le \delta$ for every $x, y \in X$. Fix $x_0 \in X$.  
Let $k_0>0$ be a natural number such that  for $k \ge k_0$, 
we have $$d\big(\phi_K(x_0), \phi_{K,k}(x_0)\big) \le \epsilon.$$
Since the line bundle $L$ is positive, $\phi_{K,r}$ is continuous 
for $r$ big enough. Hence without loss of generality 
we can assume that $\phi_{K,r}$ is continuous for every $r$, 
for only big $r$ matters for us.
By shrinking $\delta$ if necessary, we obtain that for every 
$1 \le r \le k_0$, one has 
$$d\big(\phi_{K,r}(x), \phi_{K,r}(y)\big) \le \epsilon$$ 
if $d(x,y) \le \delta$. Write $k= k_0 r +s$ for $0 \le s \le k_0 -1$. 
Using this and (\ref{ine-subaditivephi_k}) yields
$$k \phi_{K,k} \ge   r k_0 \phi_{K, k_0}+ s \phi_{K,s}.$$
It follows that 
$$\phi_{K,k} \ge  r\frac{k_0}{k} \,  \phi_{K, k_0}+
\frac{s}{k} \, \phi_{K,s}\,.$$
Thus 
$$\phi_{K,k}(x) - \phi_K(x) \ge 
\frac{r k_0}{k}(\phi_{K, k_0}(x)- \phi_K(x)) - 
\Big(1- r \frac{k_0}{k}\Big) \phi_K(x) + \frac{s}{k} \, \phi_{K,s}(x).
$$
The choice of $\delta$ now implies that 
there exists a constant $C>0$ such that  the right-hand side 
is bounded from below by $-3\epsilon- C k_0/k$ if $d(x,x_0) \le \delta$. 
Since $\phi_{K,k} \le \phi_K$, we obtain the uniform convergence 
of $\phi_{K,k}$ to $\phi_K$. 
\endproof

Put 
$$\tilde{\phi}_{K,k}:= \frac{1}{2k} \log \tilde{B}_k=
\frac{1}{2k} \log \sum_{j=1}^{d_k} |s_j|^2_{h_0^k}.$$


\begin{proposition} \label{pro-uniforBergmann} 
Assume that $(K,\phi)$ is regular and $(K,\mu,\phi)$ 
satisfies the Bernstein-Markov property. Then we have 
\begin{equation}\label{eq:fiKk}
\big\| \tilde{\phi}_{K,k} - \phi_K\big\|_{\cali{C}^0(X)} \to 0,
\:\:\text{as $k \to \infty$}.
\end{equation}
In particular, 
\begin{equation}\label{eq:B_k}
\lim_{k \to \infty} \tilde{B}_k^{1/k}= e^{2\phi_K}.
\end{equation}
\end{proposition}

Note that the limit in \eqref{eq:B_k} 
is independent of $\mu$.  
We refer to \cite[Lemma 2.8]{Beckermann-Putinar-Saff-Sty} 
for more informations in the case $K \subset \C^n \subset X =\P^n$.

\proof  When $X= \P^n$ and $L=\mathcal{O}(1)$, 
this is Lemma 3.4 in \cite{Bloom-Shiffman}.  
The arguments there work for our setting. 
We reproduce here the proof for the readers' convenience. 
It suffices to check the first desired property \eqref{eq:fiKk}.
Observe that 
\begin{align}\label{ine-BkBergman}
\sup_K (|s|^2_{h_0^k} e^{- 2k \phi})=
\sup_K |s|^2_{h^k} \le \big(\sup_K B_k \big) \|s\|^2_{L^2(\mu, k \phi)}.
\end{align}
Combining this with the  Bernstein-Markov property, 
we see that  for every $\epsilon>0$, there holds 
\begin{align}\label{ine-BMvietlai}
\sup_K (|s|^2_{h_0^k} e^{- 2k \phi})\le e^{\epsilon k} \|s\|^2_{L^2(\mu, k \phi)}
\end{align}
for every $s \in H^0(X, L^k)$. Observe also that 
\begin{align} \label{ine-sh0wahsline}
|s|_{h_0^k} \le \sup_K (|s|_{h_0^k} e^{- k \phi}) e^{k \phi_{K,k}}
\end{align}
on $X$. Applying the last inequality to $s:= s_j$ and using \eqref{ine-BMvietlai}
we infer that 
$$\frac{1}{2k}\log\tilde{B}_k\le \epsilon + \phi_{K,k}.$$
In other words, $\tilde{\phi}_{K,k} \le \epsilon + \phi_{K,k}$ on $X$. 
On the other hand, if $\sup_K (|s|_{h_0^k} e^{- k \phi}) \le 1$, then 
$$\|s\|_{L^2(\mu, k \phi)} \le C$$
for some constant $C$ independent of $k$. It follows that 
$$\tilde{B}_k \ge C^{-1} e^{2 k \phi_{K,k}}.$$
Consequently $\tilde{\phi}_{K,k} \ge \phi_{K,k}+ O(k^{-1})$. 
Thus using Lemma \ref{le-hoitudeuenvelop} we obtain the desired assertion. 
This finishes the proof. 
\endproof

\begin{remark} \label{re-notlowerbound} Recall  $\phi_K \le \phi$ on $K$. If $x \in K$ is a point so that $\phi_K(x)< \phi(x)$, then by  Proposition \ref{pro-uniforBergmann} we see that $B_k^{-1}(x)$ grows exponentially as $k \to \infty$. Consider now the case where $K=X$ and $\phi$ is not an $\omega$-psh function. In this case there exists $x \in X$ with $\phi_X(x)< \phi(x)$, and hence $B_k$ becomes exponentially small as $k \to \infty$.   
\end{remark}

\subsection{H\"older regularity of extremal plurisubharmonic envelopes}

Let $\alpha \in (0,1]$ and $Y$ be a metric space.  We denote by $\mathcal{C}^{0,\alpha}(Y)$ the space of functions on $Y$ of finite $\mathcal{C}^{0,\alpha}$-norm. If $0<\alpha<1$, then we also write $\mathcal{C}^\alpha$ for $\mathcal{C}^{0,\alpha}$. The following notion introduced in \cite{DMN} will play a crucial role for us. 

\begin{definition} For $\alpha \in (0,1]$ and $\alpha' \in (0,1],$ a  non-pluripolar compact $K$ is said to be $(\mathcal{C}^{0,\alpha}, \mathcal{C}^{0,\alpha'})$-regular if 
for any positive constant $C,$ the set $\{\phi_K: \phi \in \mathcal{C}^{0,\alpha}(K) \text{ and } \|\phi\|_{\mathcal{C}^{0,\alpha}(K)} \le C\}$ is a bounded subset of $\mathcal{C}^{0,\alpha'}(X).$
\end{definition}

The following provides examples for the last notion.

\begin{theorem}[{\cite[Theorem 2.3]{Vu_feketepoint}}]
\label{the1CalphaClapharegu} 
Let $\alpha$ be 
an arbitrary number in $(0,1).$ Then any  compact generic 
nondegenerate $\mathcal{C}^5$ piecewise-smooth submanifold $K$ 
of $X$ is $(\mathcal{C}^{0,\alpha}, \mathcal{C}^{0,\alpha/2})$-regular.  
Moreover if $K$ has no singularity, then  $K$ is 
$(\mathcal{C}^{0,\alpha}, \mathcal{C}^{0,\alpha})$-regular.
\end{theorem}

The following remark follows  immediately from  Proposition \ref{pro-hopcuaK} and the proof of \cite[Theorem 2.3]{Vu_feketepoint}.

\begin{remark}  If $K$ is as in Proposition \ref{pro-hopcuaK}, then  $K$ is  also $(\mathcal{C}^{0,\alpha}, \mathcal{C}^{0,\alpha})$-regular for $0< \alpha <1$. The union of a finite number of open subsets with smooth boundary in $X$ is an example of such $K$.
\end{remark}

If $K=X$, and $\phi \in \mathcal{C}^{0,1}$, then it was shown in 
\cite{Chu-Zhou-optimal-envelope} that $\phi_X \in \mathcal{C}^{0,1}$, 
hence $X$ is $(\mathcal{C}^{0,1}, \mathcal{C}^{0,1})$-regular; 
see also \cite{Berman-C11regula,Chu-Zhou-optimal-envelope,Tosatti-envelop}
for more information. In the case where $K=X$ or $K$ is an open subset 
with smooth boundary in $X$ it was proved in \cite{DMN} 
that $K$ is $(\mathcal{C}^{0,\alpha}, \mathcal{C}^{0,\alpha})$-regular 
for $0< \alpha <1$. This was extended for $K$ as in the statement of
Theorem \ref{the1CalphaClapharegu} in \cite[Theorem 2.3]{Vu_feketepoint}; 
see also \cite{Lu-To-Phung}. We don't know if 
Theorem \ref{the1CalphaClapharegu} holds for $\alpha=1$. 
Here is a partial result whose proof is exactly as
of \cite[Theorem 2.3]{Vu_feketepoint} by using 
\eqref{ine_passagetogeneralcase4} instead of \eqref{ine_passagetogeneralcase4cu}
(and noting that the analytic disc in Proposition \ref{pro_familydiscK} 
is $\mathcal{C}^2(\overline\D)$, hence in particular, 
is $\mathcal{C}^{1,\delta}$ for some $\delta \in (0,1]$).

\begin{theorem} \label{the1CalphaClapharegunanchinhquy} 
Let $\delta \in (0,1), C_1>0$ be constants. 
Let $K$ be a compact generic $\mathcal{C}^5$ smooth submanifold
(without boundary) of $X$. Then there exists a constant $C_2>0$
such that for every $\phi \in \mathcal{C}^{1, \delta}(K)$ with
$\|\phi\|_{\mathcal{C}^{1,\delta}} \le C_1$, then 
$\|\phi_K\|_{\mathcal{C}^{0,1}} \le C_2$.  
\end{theorem}

We mention at this point an example in \cite{Sadullaev-example}
of a domain $K$ with $\mathcal{C}^0$ boundary but $(K,\phi)$ is not regular even 
for $\phi =0$.  Applying Theorem \ref{the1CalphaClapharegunanchinhquy} 
to $X= \P^n$ we obtain the following result that 
implies \cite[Conjecture 6.2]{Sadullaev-Zeriahi} as a special case.

\begin{theorem} \label{the1CalphaClaphareguCn} 
Let $\alpha \in (0,1]$. Let $K$ be a compact generic nondegenerate 
$\mathcal{C}^5$ piecewise-smooth submanifold in $\C^n$. Let 
$$V_K:= \sup \{\psi \in \mathcal{L}(\C^n): \psi \le 0 \, \text{ on } K\}.$$
Then $V_K \in \mathcal{C}^{1/2}(\C^n)$. 
Additionally if $K$ has no singularity, then $V_K \in \mathcal{C}^{0,1}(\C^n)$. 
\end{theorem}

We note that the fact that $V_K \in \mathcal{C}^{0,1}(\C^n)$ when $K$ 
has no singularity was proved in \cite{Sadullaev-Zeriahi} 
(and as can be seen from the above discussion, this property also 
follows essentially from \cite{Vu_feketepoint}).
As a direct consequence of Theorem \ref{the1CalphaClaphareguCn}, 
we record here a Bernstein-Markov type inequality of independent interest. 
We will not use it anywhere in the paper.

\begin{theorem}\label{the-Bernstein-totally-real} 
Let $K$ be a compact generic nondegenerate 
$\mathcal{C}^5$ piecewise-smooth submanifold in $\C^n$. 
Then there exists a constant $C>0$ such that for every 
complex polynomial $p$ on $\C^n$ we have
\begin{equation}\label{eq:grad1}
\|\nabla p\|_{L^\infty(K)} \le C (\deg p)^{2}\|p\|_{L^\infty(K)}.
\end{equation} 
If additionally $K$ has no singularity (e.g, $K= \S^{2n-1}$), then 
\begin{equation}\label{eq:grad2}
\|\nabla p\|_{L^\infty(K)} \le C \deg p \,\|p\|_{L^\infty(K)}.
\end{equation}  
\end{theorem} 

Note that the exponent of $\deg p$ is optimal as it is well-known for
the classical Markov and Bernstein inequalities in dimension one. 
The above result was known when $K$ is algebraic in $\R^n$, 
see \cite{Berman-OrtegaCerda,Bos-Levenberg-Milman-Taylor}. 
We underline that inequalities similar to those in 
Theorem \ref{the-Bernstein-totally-real} also hold for other situations 
(with the same proof), for example, $K= \S^{n-1} \subset \R^n$ 
and considering $K$ as a maximally totally 
real submanifold in the complexification of $\S^{n-1}$. 
 
Markov (or Bernstein) type inequalities are a subject of great interest in 
approximation theory. There is a large literature on this topic, e.\,g.\ 
\cite{Berman-OrtegaCerda,Bos-Levenberg-Milman-Taylor,
Bos-Levenberg-Milman-Taylor2,Brudnyi-bernstein-ine-sub,
Brudnyi,Coman-Poletsky,Pierzchala,Yomdin-smooth-parametrization}, 
to cite just a few.

\proof  Let $$\phi_K:= \sup\{\psi\in\mathcal{L}(\C^n):\sup_K \psi \le 0\}$$
which is $\mathcal{C}^{1}$ if $K$ has no singularity or
$\mathcal{C}^{1/2}$ in general  by Theorem \ref{the1CalphaClaphareguCn}. 
Let $p$ be a complex polynomial in $\C^n$. Put $k:= \deg p$. 
Since $\frac{1}{k} (\log |p|- \log \max_K|p|)$ is a candidate 
in the envelope defining $\phi_K$, we get
$$|p| \le e^{k \phi_K} \max_K |p|$$
on $\C^n$. 
We use the same notation $C$ to denote a constant depending only on $K,n$.  
Let $a=(a_1,\ldots,a_n) \in K \subset \C^n$. Let $r>0$ be a small constant.  
Consider the analytic disc $D_a:= (a_1+ r\D, a_2,\ldots, a_n)$. 
Applying the Cauchy formula to the restriction of $p$ to $D_a$ shows that 
$$|\partial_{z_1}p(a)| \le r^{-1} \max_{D_a} |p|\le
r^{-1} (\max_K |p|) \max_{D_a} e^{k \phi_K}.$$
Since $\phi_K = 0$ on $K$, using $\mathcal{C}^{1/2}$ regularity
of $\phi_K$ gives
$$|\partial_{z_1}p(a)| \le r^{-1} \max_{D_a} |p|\le 
r^{-1} (\max_K |p|) e^{C k r^{1/2}}$$
for some constant $C>0$ independent of $p$ and $a$. 
Choosing $r= k^{-2}$ in the last inequality yields
$$|\partial_{z_1}p(a)| \le C k^{2} \max_{D_a} |p|.$$
Similarly we also get  $$|\partial_{z_j}p(a)| \le 
C k^{2} \max_{D_a} |p|$$
for every $1 \le j \le n$. Hence the first desired inequality
\eqref{eq:grad1} 
for general $K$. When $K$ has no singularity, the arguments are similar. 
This finishes the proof.  
\endproof

Here is a quantitative version of Lemma \ref{le-pointwiseconvergence}.
 
\begin{proposition} \label{pro-hoituC0phiKkphiK}  
Let $K$ be a compact generic nondegenerate $\mathcal{C}^5$ 
piecewise-smooth submanifold of $X$.  Let $\phi$ be a 
H\"older continuous function on $K$. Then, we have 
$$\big\| \phi_{K,k}- \phi_K\big\|_{\cali{C}^0(X)}=
O\left(\frac{\log k}{k}\right).$$
\end{proposition}
 
\proof The desired estimate was proved for $K=X$ in 
\cite[Corollary 4.4]{Dinh-Ma-Marinescu}. 
For the general case, we use the proof of 
\cite[Theorem 2.3]{Vu_feketepoint} (and \cite{DMN}). 
Let $\tilde{\phi}$ be the continuous extension of $\phi$ 
to $X$ as in the proof of  \cite[Theorem 2.3]{Vu_feketepoint}. 
It was showed there that $\phi_K= \tilde{\phi}_X$. 
On the other hand, one can check directly that 
$\phi_{K,k} \ge \tilde{\phi}_{X,k}$. 
Hence, we get
$$| \phi_{K,k}- \phi_K| = \phi_K - \phi_{K,k}  
\le \tilde{\phi}_X - \tilde{\phi}_{X,k}=
O\left(\frac{\log k}{k}\right),$$
for every $k$. This implies the conclusion.
 \endproof

\begin{proof}[End of the proof of Theorem \ref{the-convergencebergmankernel}] The desired estimate is deduced directly by using Proposition \ref{pro-hoituC0phiKkphiK}, Theorem \ref{the-strongBMintro}, and following the same arguments as in the proof of Proposition \ref{pro-uniforBergmann}. We just briefly recall here how to do it. Firstly as in the proof of Proposition \ref{pro-uniforBergmann}, we have 
$$\tilde{\phi}_{K,k}- \phi_{K,k} \ge O(k^{-1}).$$
It remains to bound from above $\tilde{\phi}_{K,k}- \phi_{K,k}$. Combining the polynomial upper bound for $B_k$ in Theorem \ref{the-strongBMintro} and (\ref{ine-BkBergman}), one gets,  for some constants $C,N>0$ independent of $k$, 
\begin{align*}
\sup_K (|s|_{h_0^k} e^{- k \phi}) \le C k^{N} \|s\|_{L^2(\mu, k \phi)}
\end{align*}
for every $s \in H^0(X, L^k)$. This coupled with (\ref{ine-sh0wahsline}) yields
\begin{align*}
|s|_{h_0^k} \le k^{N} e^{k \phi_{K,k}}
\end{align*}
on $X$. It follows that 
$$\tilde{\phi}_{K,k}=\frac{1}{2k}\log \tilde{B}_k \le
\phi_{K,k}+ N\,\frac{\log k}{k}.$$
This finishes the proof. 
\end{proof}

\section{Polynomial growth of Bergman kernel functions}

This section is devoted to the proof of 
Theorems \ref{the-bergman-smooth} and \ref{the-strongBMintro}. 

\subsection{Families of analytic discs attached to $K$}

The goal of this part is to construct suitable families of 
analytic discs partly attached to $K$. 
This is actually implicitly contained 
in \cite{Vu_MA}. We don't need all of properties of the family 
of analytic discs given in \cite{Vu_MA}. 
For readers' convenience we recall briefly the construction below.
We will only consider the case where $\dim K=n$ in this section.   

Here is our result giving the desired family of analytic discs.

\begin{theorem} \label{th-discnoboundary} 
Let $C_0>0$ be a constant. Then there exist constants $C>0,r_0>0$
and $\theta_0 \in (0,\pi/2)$ such that for every $0<t < r_0$, 
the following properties are satisfied.  
Let $p_0$ be a regular point of $K$ of distance at least $t/C_0$ 
to the singularity of $K$, and let $W_{p_0}$ be a local chart around 
$p_0$ in $X$ such that $p_0$ corresponds to the origin $0$ in $\C^n$. 
Then there is a $\mathcal{C}^2$ map 
$F: \overline \D \times \B_{n-1} \to W_{p_0}$ 
such that the following properties are fulfilled:

(i) $F(\cdot,y)$ is holomorphic for every $y \in \B_{n-1}$, and 
$$ \|DF(\xi, y)\| \le C t,$$
for every $\xi \in \D$, $y \in \B_{n-1}$.

(ii) $F(e^{i\theta},y) \in K$ for $-\theta_0 \le \theta \le \theta_0$ and $y \in \B_{n-1}$, and 
$$|F(1,y)| \le t/C_0,$$

(iii) Let $G$ denote the restriction of $F$ to $[e^{-i \theta_0}, e^{i \theta_0}] \times \B_{n-1}$. Then $G$ is bijective onto its image, and the image of $G$ is contained in $\B(p_0, t/C_0) \cap K$ (here $\B(p_0,t/C_0)$ denotes the ball centred at $p_0$ of radius $t/C_0$ in $X$),  and  
$$C^{-1} t^{n} \le \big|\det DG(e^{i\theta}, y)\big| \le C t^{n},$$     
for every $- \theta_0 \le \theta \le \theta_0$ and $y \in \B_{n-1}$.
\end{theorem}

 We proceed with the proof of the last theorem.
Denote by $z=x+ i y$ the complex variable on $\C$ and by $\xi=e^{i\theta}$ the variable on $\partial \D.$ For any $m \in \N$ and $r >0,$ let $\B_m(0,r)$ be the Euclidean ball centered at $0$ of radius $r$ of $\R^m$, and for $r=1$, we write $\B_m$ for $\B_m(0,1)$.  
Let $Z$ be  a compact submanifold with or without boundary of $\R^m.$ The Euclidean metric on $\R^m$ induces a metric on $Z.$  For $\beta \in (0,1]$ and $k \in \N$, let $\mathcal{C}^{k,\beta}(Z)$ be the space of real-valued functions on $Z$ which are differentiable up to the order $k$ and whose $k^{th}$ derivatives are H\"older continuous of order $\beta.$   For any tuple $v=(v_0,\cdots,v_m)$ consisting of functions in  $\mathcal{C}^{k,\beta}(Z)$, we define its $\mathcal{C}^{k,\beta}$-norm to be the maximum of the ones of  its components.

Let $u_0$ be a continuous function on $\partial \D$. 
Let 
$$\mathcal{C}u_0(z):= \frac{1}{2\pi} \int_{-\pi}^{\pi} u_0(e^{i\theta}) \frac{e^{i\theta}+ z}{e^{i\theta}-z} d\theta$$
which is a holomorphic function on $\D$. Recall that the real part of $\mathcal{C} u_0$ is $u_0$. Let $\mathcal{T}u_0(z)$ denotes the imaginary part of $\mathcal{C}u_0(z)$.    Put
 $$\mathcal{T}_1 u_0:= \mathcal{T}u_0 - \mathcal{T}u_0(1).$$
For basic properties of $\mathcal{T}_1$, one can consult \cite{MerkerPorten2} or \cite{Baouendi_Ebenfelt_Rothschild}.

We now go back to our current situation with $X$. We endow $X$ with an arbitrary Riemannian metric. For $\theta_0 \in (0,\pi)$, let $[e^{-i \theta_0}, e^{i \theta_0}]$ denotes the arc of $\partial \D$ of arguments from $-\theta_0$ to $\theta_0$.
Let $p_0$ be a regular point in $K$ and let $r_{p_0}$ denotes the distance of $p_0$ to the singular part of $K$. Recall that we assume in this section that $\dim K=n$.  

\begin{lemma} \label{le_localcoordinates} There exist   a constant $c_K>1$ depending only on $(K,X)$ and a local chart $(W_{p_0}, \Psi)$ around $p_0,$ where $\Psi: W_ {p_0} \rightarrow \B_{2n}$ is biholomorphic with $\Psi(p_0)=0$  such that the two following conditions hold:  

$(i)$ we have 
\begin{align*} 
\|\Psi\|_{\mathcal{C}^{5}} \le c_K, \quad \|\Psi^{-1}\|_{\mathcal{C}^{5}} \le c_K,
\end{align*}

$(ii)$ there is a $\mathcal{C}^{3}$ map $h$ from $\overline{\B}_n$ to $\R^n$ so that $h(0)=Dh(0)=0$, 
and
 $$\Psi(K \cap W_{p_0}) \supset  \big\{(\mathbf{x}, h(\mathbf{x})): \mathbf{x} \in \overline{\B}_n(0, r_{p_0}/c_K)\big\},$$
 where the canonical coordinates on $\C^n= \R^n + i \R^n$ are denoted by $\mathbf{z}= \mathbf{x}+ i \mathbf{y},$ and 
\begin{align} \label{ine_C1alphanorm}
\|h\|_{\mathcal{C}^3} \le c_K.
\end{align}
\end{lemma}

Note that $h$ is indeed $\mathcal{C}^5$ (because $K$ is so), but $\mathcal{C}^3$ is sufficient for our purpose in what follows. 

\proof The existence of local coordinates so that $h(0)=Dh(0)=0$ is standard, 
see \cite{BER} or \cite[Lemma 4.1]{Vu_feketepoint}. 
Perhaps one needs to explain a bit about the radius $r_{p_0}/c_K$: 
the existence of $c_K$ comes from the fact that for every singular point 
$a$ in $K$, there are an open  neighorhood $U$ of $a$ in $X$ 
and sets $A_j \Subset B_j \subset U$ for $1 \le j \le m$ 
such that $B_j$ is $\mathcal{C}^5$ smooth generic CR submanifold 
of dimension $\dim K$ in $U$, and  
$$K\cap U=\cup_{j=1}^m (A_j\cap U),$$
 and $A_j$ is the closure in $U$ of a relatively compact 
 open subset in $B_j$, and $\partial A_j$ (in $U$) 
 is contained in the singularity of $K$. 
 Thus by applying the standard local coordinates to points 
 in $B_j$ we obtain the existence of $c_K$. This finishes the proof.  
\endproof

From now on, we only use the local coordinates introduced in Lemma \ref{le_localcoordinates} and identify points in $W_{p_0}$ with those in $\B_{2n}$ via $\Psi.$ 

\begin{lemma}[{\cite[Lemma 3.1]{Vu_feketepoint}}]
\label{le_version8existenceu_MA}
There exist a function 
$u_0 \in \mathcal{C}^{\infty}(\partial \D)$ such that 
$u_0(e^{i\theta}) = 0$ for  $\theta \in [-\pi/2, \pi/2]$ and
$\partial_x u_0(1)=-1$.
 \end{lemma}

In what follows, we identify $\C^n$ with $\R^n+ i \R^n.$ Let  $u_0$ be a function described in Lemma \ref{le_version8existenceu_MA}.  Let $\boldsym{\tau}_1, \boldsym{\tau}_2 \in \overline{\B}_{n-1} \subset \R^{n-1}.$ Define $\boldsym{\tau}^*_1:= (1, \boldsym{\tau}_1) \in \R^n$ and $\boldsym{\tau}^*_2:= (0, \boldsym{\tau}_2) \in \R^n$ and $\boldsym{\tau}:=(\boldsym{\tau}_1,\boldsym{\tau}_2).$ Let $t$ be a positive number in $(0,1]$ which plays a role as a scaling parameter in the equation (\ref{Bishoptype}) below. 

In order to construct an analytic disc partly attached to $K$, it suffices to find a  map 
$$U: \partial \D \rightarrow \B_n \subset \R^n,$$
which is H\"older continuous, satisfying the following Bishop-type equation 
\begin{align}\label{Bishoptype}
U_{\boldsym{\tau},t}(\xi)= t\boldsym{\tau}^*_2 - \mathcal{T}_1\big(h(U_{\boldsym{\tau},t}) \big)(\xi) - t\mathcal{T}_1 u_0(\xi) \, \boldsym{\tau}^*_1,
\end{align}      
where the Hilbert transform $\mathcal{T}_1$ is extended to a vector valued function by acting on each component.  The existence of solution of the last equation is a standard fact in the Cauchy-Riemann geometry. 

\begin{proposition}[{\cite[Proposition 3.3]{Vu_MA}}]
\label{pro_BishopequationMA}
There  are a positive number $t_1 \in (0,1)$ and a real number 
$c_1>0$  satisfying the following property: for any $t\in(0,t_1]$ 
and any $\boldsym{\tau} \in \overline{\B}_{n-1}^2,$ 
the equation (\ref{Bishoptype}) has a unique solution 
$U_{\boldsym{\tau},t}$ which is  
$\mathcal{C}^{2, \frac{1}{2}}$ in $(\xi, \boldsym{\tau})$ 
and such that
\begin{align}\label{ine_danhgiachuancuaU}
\|D^j_{(\xi,\boldsym{\tau})} 
U_{\boldsym{\tau},t}\|_{\mathcal{C}^{\frac{1}{2}}(\partial \D)} 
\le c_1  t,
\end{align} 
for any $\boldsym{\tau} \in \overline{\B}_{n-1}^2$ and $j=0,1$.
\end{proposition}

From now on, we consider $t<\min \{t_1, C_0 r_{p_0}\}$ (hence the distance from $p_0$ to the singularity of $K$ is at least $t/C_0$).  Let $U_{\boldsym{\tau},t}$ be the unique solution of (\ref{Bishoptype}).
For simplicity,  we use the same notation $U_{\boldsym{\tau},t}(z)$ to denote the harmonic extension of $U_{\boldsym{\tau},t}(\xi)$ to $\D.$  Let $P_{\boldsym{\tau},t}(z)$ be  the harmonic extension of $h\big(U_{\boldsym{\tau},t}(\xi)\big)$ to $\D.$ Recall that 

\begin{lemma}\label{le_danhgiadaohamcuaP'tauMA} (\cite[Lemma 3.4]{Vu_MA}) There exists a constant $c_2$ so that for every $t \in (0, t_1]$ and every $(z,\boldsymbol{\tau}) \in \overline{\D} \times \overline{\B}_{n-1}^2,$ we have
\begin{align}\label{ine_danhgiachuancuaP'2}
\| D^j_{(z,\boldsym{\tau})} U_{\boldsymbol{\tau},t}(z)\| \le c_2 t \quad  \text{and} \quad  \| D^j_{(z,\boldsym{\tau})} P_{\boldsymbol{\tau},t}(z)\| \le c_2 t^2,
\end{align} 
for $j=0,1.$
\end{lemma}

We note that the hypothesis that $D^2h(0)=0$ was required in \cite{Vu_MA}, but it is actually superfluous in the proof of Lemma \ref{le_danhgiadaohamcuaP'tauMA}.  Define
$$F(z, \boldsym{\tau},t) := U_{\boldsym{\tau},t}(z)+ i  P_{\boldsym{\tau},t}(z)+ i  t \, u_0(z) \, \boldsym{\tau}^*_1$$
which is a family of analytic discs parametrized by $(\boldsym{\tau},t).$ Compute 
$$F(1,\boldsym{\tau},t)= U(\boldsym{\tau},t)(1)+ i  P_{\boldsym{\tau},t}(1)= t\boldsym{\tau}^*_2+ h(t\boldsym{\tau}^*_2).$$
Hence if $|\boldsym{\tau}_2|$ is small enough, we see that 
\begin{align}\label{ine-Ftai1nho}
|F(1,\boldsym{\tau},t)| \le t |\boldsym{\tau}_2|< r_{p_0}/(2c_K).
\end{align}
This combined with Lemma \ref{le_danhgiadaohamcuaP'tauMA} yields that 
\begin{align}\label{ine-Ftai1nho2}
|U_{\boldsym{\tau},t}(e^{i\theta})| \le |\theta||U_{\boldsym{\tau},t}(1)| \le |\theta| r_{p_0}/(2c_K)<r_{p_0}/c_K
\end{align}
if $\theta$ and $\boldsym{\tau}_2$ are small enough. Now the defining formula of $F$ and the fact that $u_0 \equiv 0$ on $ [e^{-i \pi/2}, e^{i\pi/2}]$ imply that  
$$F(\xi, \boldsym{\tau},t)=U_{\boldsym{\tau},t}(\xi)+ i  P_{\boldsym{\tau},t}(\xi)=U_{\boldsym{\tau},t}(\xi)+ i h\big(U_{\boldsym{\tau},t}(\xi)\big) \in  K$$
by (\ref{ine-Ftai1nho2}) if $\theta$ and $\boldsym{\tau}_2$ are small enough. 
In other words, there is a small constant $\theta_0$ such that if $|\boldsym{\tau}_2|<\theta_0$, then  $F$ is $[e^{-i\theta_0}, e^{i\theta_0}]$-attached to $K$. 


\begin{proposition} \label{pro_corverKtau1fixed} By decreasing $\theta_0$ and $t_1$ if necessary, we obtain the following property: for every $\boldsym{\tau}_1 \in \overline{\B}_{n-1},$  the map $F(\cdot, \boldsym{\tau}_1,\cdot,t): [e^{-i\theta_{0}}, e^{i\theta_{0}}] \times \overline{\B}_{n-1}(0, \theta_0) \rightarrow K$ is a diffeomorphism onto its image, and 
$$ C^{-1} t^n \le \|\det DF(\cdot, \boldsym{\tau}_1,\cdot,t)\|_{L^\infty} \le C t^n,$$
for some constant $C>0$ independent of $t, \boldsym{\tau}_1$.   
\end{proposition}
  
\proof The desired assertion was implicitly obtained in the proof of  \cite[Proposition 3.5]{Vu_MA}. We present here complete arguments for readers' convenience. Recall $$F(e^{i\theta}, \boldsym{\tau}_1,t)= U_{\boldsym{\tau},t}(e^{i\theta})+i h(U_{\boldsym{\tau},t}(e^{i\theta})$$
By the Cauchy-Riemann equations, we have 
$$\partial_y U_{\boldsym{\tau},t}(1)= -  t \partial_x u_0(1) \boldsym{\tau}^*_1-  \partial_x P_{\boldsym{\tau},t}(1)= t \boldsym{\tau}^*_1-  \partial_x P_{\boldsym{\tau},t}(1).$$   
The last term is $O(t^2)$ by Lemma \ref{le_danhgiadaohamcuaP'tauMA}. Thus the first component of $\partial_y U_{\boldsym{\tau},t}(1)$ is greater than $t/2$ provided that $t \le t_2$ small enough. A direct computation gives  $\partial_y U_{\boldsym{\tau},t}(1)=\partial_{\theta} U_{\boldsym{\tau},t}(1).$ Consequently, the first component of 
$$\partial_\theta F(e^{i\theta}, \boldsym{\tau}_1,t)= \partial_{\theta} U_{\boldsym{\tau},t}(1)+ i \partial_{\theta} h\big(U_{\boldsym{\tau},t}(1)\big)$$ is greater than $t/2$ for $t \le t_2$ (note that $Dh(0)=0$).  Moreover, as computed above,  we have 
$$F(1, \boldsym{\tau}_1,t)= t \boldsym{\tau}_2^*+ h(t \boldsym{\tau}_2^*).$$ Thus $D_{\boldsym{\tau}_2, \theta} F(1, \boldsym{\tau},t)$ is a nondegenerate matrix whose determinant satisfies the desired inequalities if $|\theta|<\theta_0$ is small enough.  The proof is finished.
\endproof

\begin{proof}[End of proof of Theorem \ref{th-discnoboundary}]

Let  $\theta_0$ be as above and smaller than $\theta_1$  and $M> |\theta_0|^{-1}$ be a big constant. Fix a parameter $\boldsym{\tau}_1$ and define 
$$F_{t}(\xi, \boldsym{\tau}_2):= F(\xi,\boldsym{\tau}_1,\boldsym{\tau}_2/M,t).$$
By the above results, we see that the family $F_t$ satisfies all of required properties (because $|\boldsym{\tau}_2/M|< \theta_0$). This finishes the proof.
\end{proof}

\subsection{Upper bound of Bergman kernel functions}

We start with the following useful estimate in one dimension.  

\begin{lemma} \label{le-hamsubtrenduongtron} Let $\beta \in (0,1)$. Then there exists a constant $C_\beta>0$ such that for every $\theta_0 \in (0, \pi]$ and every constant $M>0$, and every subharmonic function $g$ on $\D$ such that $g$ is continuous up to $\partial \D$ and $|g(\xi)| \le M$ for $\xi \in \partial \D \backslash \{e^{i\theta}: -\theta_0 \le \theta \le \theta_0\} $,  we have 
$$g(z)\le C_\beta \bigg[ |1- z|^\beta \theta_0^{-1}M + (1- |z|)^{-1} \int_{-\theta_0}^{\theta_0} g(e^{i\theta}) d\theta\bigg].$$  
\end{lemma}

\proof Let $\theta_1 \in [0, 2 \pi)$. Put 
$$I:=\{e^{i\theta}: -\theta_0 \le \theta \le \theta_0\}, \quad  I':= \{e^{i\theta}: -\theta_0/4 \le \theta \le \theta_0/4\}.$$
Let $g_1$ be the harmonic function on $\D$ such that $g_1 \in \mathcal{C}^{0,1}(\partial \D)$, and  $g_1(\xi)= M$ for $\xi \in \partial \D \backslash I $ and $g_1 \equiv 0$ on $I'$. Observe that $\| g_1\|_{\mathcal{C}^{0,1}(\partial \D)} \le C\theta_0^{-1} M$ for some constant $C>0$ independent of $M,\theta_1,\theta_0$. 

By a  classical result  on harmonic functions on the unit disc (see \cite[Page 41]{Gakhov} or  (3.4) in \cite{Vu_feketepoint}), we have  $$\|g_1\|_{\mathcal{C}^{0,\beta}(\D)} \lesssim  \| g_1\|_{\mathcal{C}^{0,\beta}(\partial \D)} \lesssim \theta_0^{-1}M$$
for every $\beta \in (0,1)$.  As a result, we get 
\begin{align}\label{ine-g1holder}
g_1(z)= g_1(z)- g_1(1) \lesssim |1-z|^\beta \|g_1\|_{\mathcal{C}^{0,\beta}(\partial \D)} \lesssim |1-z|^\beta \theta_0^{-1} M.
\end{align}
Let $g_2$ be the harmonic function on $\D$ such that $g_2(e^{i \theta}) =0$ for  $|\theta|> \theta_0$, and $g_2(e^{i\theta})= g(e^{i\theta})$ for $\theta \in [-\theta_0, \theta_0]$. Observe that $g \le g_1+ g_2$ because the latter function is harmonic and greater than or equal to $g$ on the boundary of $\D$. Using Poisson's formula, we see that
$$g_2(z) \le  (1-|z|)^{-1}  \int_{-\pi}^{\pi} g_2(e^{i\theta}) d\theta=  (1-|z|)^{-1}  \int_{-\theta_0}^{\theta_0} g(e^{i\theta}) d\theta.$$
Summing this and  (\ref{ine-g1holder}) gives the desired assertion.  The proof is finished.
\endproof

Let $K$ be a $\mathcal{C}^5$ smooth (without boundary) maximally totally real submanifold in $X$. Let $s \in H^0(X,L^k)$ with $\|s\|_{L^2(\mu, h^k)} =1$ and $M:= \sup_K |s|^2_{h^k}$. 
Let $p_0\in K$. Consider a local chart $(U,\mathbf{z})$ around $p_0$ with coordinates $\mathbf{z}$, 
and  $p_0$ corresponds to the origin $0$ in $\C^n$.  Shrinking $U$ if necessary we can assume also that   $L$ is trivial on $U$.

We trivialize $(L,h_0)$ over $U$ such that $h_0= e^{-\psi}$ for some psh function $\psi$ on $U$ with $\psi(0)=- \phi(0)$ (we implicitly fix a local holomorphic frame on $L|_U$ so that one can identify Hermitian metrics on $L|_U$ with functions on $U$), and identify $s$ with a holomorphic function $g_s$ on $U$. Thus 
$$h= e^{-\phi} h_0= e^{-\phi- \psi}$$ on $U$. In particular $h(0)=1$ on $U$.

\begin{lemma} \label{le-Markovinequalitysection}
There exists a  constant $C_1>0$ independent of $k,s,p_0$ such that
\begin{align}\label{ine-s-BWsm0}
\sup_{\{|\mathbf{z}|\le 1/k\}} |s(\mathbf{z})|^2_{h^k} \le C_1 M,
\end{align}
and
\begin{align}\label{ine-s-BWsm}
C_1^{-1} |s(\mathbf{z})|_{h^k}^2 \le  |g_s(\mathbf{z})|^2 \le C_1| s(\mathbf{z})|_{h^k} \le  C_1^2 M,
\end{align}
for $\mathbf{z} \in \B(0, k^{-1})$. Moreover  
 for every constant $\epsilon >0$, there exist a constant $c_\epsilon>0$ independent of $k,s,p_0$ such that 
$$|g_s(0)|^2 \le |g_s(\mathbf{z})|^2+ \epsilon M,$$   
for $|\mathbf{z}| \le 1/ (c_\epsilon k)$.
\end{lemma}

\proof The desired inequality (\ref{ine-s-BWsm})  follows immediately from (\ref{ine-s-BWsm0}) and the equalities $$ |g_s(\mathbf{z})|^2= |s(\mathbf{z})|_{h^k}^2 e^{k (\psi(\mathbf{z})+ \phi(\mathbf{z}))}, \quad \psi(0)+ \phi(0)=0.$$
Recall that 
$$\phi_K:= \sup\{\psi \quad \text{$\omega$-psh}: \psi \le \phi \quad \text{on } \quad K\}.$$
 Note that $\phi_K \le \phi$ on $K$. By hypothesis $\phi \in \mathcal{C}^{1,\delta}(K)$ for some constant $\delta>0$. This combined with Theorem \ref{the1CalphaClapharegunanchinhquy} and  the fact that  $K$ has no singularity  yields that $\phi_K$ is Lipschitz.  Using the fact that $k^{-1}\log |s|_{h^k_0}$ is $\omega$-psh, we get the Bernstein-Walsh inequality
$$|s|_{h_0^k}^2 \le  \big(\sup_K |s|^2_{h^k} \big) e^{2k \phi_K}= M e^{2k \phi_K}.$$
Hence 
$$|s|^2_{h^k} \le M e^{2k (\phi_K-\phi)}.$$
This combined with the Lipschitz property of $\phi_K$ and $\phi$ (and also the property that $\phi_K(0)- \phi(0) \le 0$ on $K$) yields 
\begin{align}\label{ine-chantrensxganK}
\sup_{\{\mathbf{z}: |\mathbf{z}|\le 1/k\}} |s(\mathbf{z})|^2_{h^k} \le C_1 M,
\end{align}
for some constant $C_1$ independent of $k,s,p_0$. Hence (\ref{ine-s-BWsm0}) also follows.

  By arguing as in the proof of Theorem \ref{the-Bernstein-totally-real}, one obtains the following version of Bernstein-Markov inequality: 
for $ \mathbf{z} \in \B(0,k^{-1}/2)$, there holds
\begin{align} \label{BMsection}
|\nabla g_s(\mathbf{z})| \lesssim k \sup_{\B(0,k^{-1})} |g_s| \lesssim  k M^{1/2}.
\end{align}
Consequently, for $ \mathbf{z} \in \B(0,k^{-1}/c_\epsilon)$ with $c_\epsilon$ big enough, we get  
$$|g_s(\mathbf{z})- g_s(0)| \le |\mathbf{z}| \sup_{\B(0, k^{-1}/C_0)} |\nabla g_s| \le \epsilon^{1/2} M^{1/2}.$$
This finishes the proof.
\endproof

\begin{proof}[Proof of the upper bound in Theorem \ref{the-bergman-smooth}] 

Let $s \in H^0(X,L^k)$ with $\|s\|_{L^2(\mu, h^k)} =1$.  Let $M:= \|s\|^2_{L^\infty(K,h^k)}$. We need to prove that $M\lesssim k^{n}$.


Let $p_0 \in K$  and consider a local chart $(U,\mathbf{z})$ around $p_0$ with coordinates $\mathbf{z}$, the point $p_0$ corresponds to $0$ in the local chart $(U,\mathbf{z})$. 
Let $\epsilon>0$ be a small constant to be chosen later. Let $C_1, c_\epsilon$ be the constants in Lemma \ref{le-Markovinequalitysection}.  Let $A \ge C_1^2 \epsilon^{-2}$ be a  big constant. Using the Lipschitz continuity of $\phi$ yields 
\begin{align}\label{ine-chanchuanL1cuagssm}
\int_{\{|\mathbf{z}| \le  k^{-1}\}} |g_s|^2 d \mu = \int_{\{|\mathbf{z}| \le  k^{-1}\}} |s|_{h^k}^2 e^{k (\psi+ \phi)} d \mu  \lesssim 1
\end{align}
(uniformly in $s$).  Now let $F: \D \times Y \to X$ be the family of analytic discs in Theorem \ref{th-discnoboundary} associated to $C_0:= 2c_\epsilon$ for $p_0,$ and $t:= k^{-1}$, where $Y:= \B_{n-1}$. 
 Let $\theta_0 \in (0, \pi/2)$ be the constant in the last theorem.

Let $g:= |g_s \circ F|^2$. Put $\mathbf{z}_y:= F(1-1/A,y)$. By expressing 
$$\mathbf{z}_y= F(1-1/A,y)- F(1,y)+ F(1,y),$$
one gets
 $$|\mathbf{z}_y| \le 1/(c_\epsilon k)$$
  if $A$ is big enough.  Applying Lemma \ref{le-hamsubtrenduongtron} to $g(\cdot,y)$ and $\beta=1/2$ and using (\ref{ine-s-BWsm}) yield
$$|g_s(\mathbf{z}_y)|^2= g(\xi_y,y) \le  C'\big( A^{-1/2} M  +  A \int_{-\theta_0}^{\theta_0} g(e^{i\theta}, y) d \theta\big),$$
for some constant $C'>0$ independent of $s,p_0,k$.   This combined with Lemma \ref{le-Markovinequalitysection} yields
$$|g_s(0)|^2 \le  \epsilon M+ C' A^{-1/2} M  +  C' A \int_{\theta_0}^{\theta_0} g(e^{i\theta}, y) d \theta.$$ 
Integrating the last inequality over $ y\in Y$ gives
\begin{align*} 
|g_s(0)|^2 &\le \epsilon M+ C' A^{-1/2} M+  C' A (\vol(Y))^{-1}\int_Y vol_y \int_{-\theta_0}^{\theta_0} g(e^{i\theta},y) d \theta\\
 &\le 2 \epsilon M+  C' A (\vol(Y))^{-1}\int_Y vol_y \int_{-\theta_0}^{\theta_0} g(e^{i\theta},y) d \theta
\end{align*}
if $A \ge \epsilon^{-2} C'^2$.  By Properties of $F$ and (\ref{ine-chanchuanL1cuagssm}), the second term in the right-hand side of the last inequality is $\lesssim k^{n}$. We infer that 
$$|g_s(0)|^2 \le 2\epsilon M+  A C_2 k^{n}$$
where $C_2$ is a constant independent of $k,s,p_0$. Consequently by (\ref{ine-s-BWsm}), one gets 
$$|s(p_0)|^2_{h^k} \le 2C_1 \epsilon  M + A C_2 C_1 k^{n}$$ for every $p_0 \in K$. By choosing $\epsilon:= 1/(4 C_1)$ and $A$ big enough as required, one gets 
$$M \le M/2+ A C_1 C_2 k^n.$$
Thus  the desired upper bound for $M$ follows. The proof is finished.     
\end{proof}

We now proceed with the proof of Theorem \ref{the-strongBMintro}.

\begin{proof}[End of proof of  Theorem \ref{the-strongBMintro} 
for  $\dim K=n$] We assume $\dim K=n$. 
We will explain how to treat the case $\dim K \ge n$ later. 
Let $s \in H^0(X, L^k)$ with $\|s\|_{L^2(\mu,h^k)}=1$. 
Put $M:= \sup_K |s|_{h^k}$. Let $\alpha \in (0,1)$ 
be a H\"older exponent of $\phi$. 
We have $\phi_K \in \mathcal{C}^{\alpha/2}$ (note $\alpha <1$). 
We follow essentially the scheme of the proof for the upper bound 
of $B_k$ in Theorem \ref{the-bergman-smooth}.

Denote by $K_{k,\alpha}$ the set of points in $K$ of distance at least $k^{-2/\alpha}$ to the singular part of $K$, and $U_{k,\alpha}$ the set of points in $X$ of distance at most $k^{-2/\alpha}$ to $K$. As in the proof of  Lemma \ref{le-Markovinequalitysection}, one has 
\begin{align}\label{ine-chantrensxganKsing}
\sup_{p \in U_{k,\alpha}} |s(p)|^2_{h^k} \le C M, \quad M \le C \sup_{p \in K_{k,\alpha}} |s(p)|^2_{h^k},
\end{align}
for some constant $C \ge 4$ independent of $k,s$. In particular it is sufficient to estimate $|s(p)|^2_{h^k}$ for $p \in K_{k,\alpha}$.

Let $p_0 \in K_{k,\alpha}$ (hence $p_0$ is a regular point of $K$, and the ball $\B(p_0, k^{-2/\alpha})\cap K$ lies entirely in the regular part of $K$) and a local chart $(U,\mathbf{z})$ around $p_0$ with coordinates $\mathbf{z}$, the point $p_0$ corresponds to $0$ in the local chart $(U,\mathbf{z})$. 

We trivialize $(L,h_0)$ over $U$ such that $h_0= e^{-\psi}$ for some smooth psh function $\psi$ on $U$ with $\psi(0)=- \phi(0)$ (we implicitly fix a local holomorphic frame on $L|_U$ so that one can identify Hermitian metrics on $L|_U$ with functions on $U$), and identify $s$ with a holomorphic function $g_s$ on $U$. Thus 
$$h= e^{-\phi} h_0= e^{-\phi- \psi}$$ on $U$. In particular $h(0)=1$ on $U$.  Using the H\"older continuity of $\phi$ yields 
\begin{align*}
\int_{\{|\mathbf{z}| \le  k^{-2/\alpha}\}} |g_s|^2 d \mu = \int_{\{|\mathbf{z}| \le  k^{-2/\alpha}\}} |s|_{h^k}^2 e^{k (\psi+ \phi)} d \mu  \lesssim 1
\end{align*}
(uniformly in $s,k$).  As in the proof of Lemma \ref{le-Markovinequalitysection}, a Berstein-Walsh type inequality implies  that 
by increasing $C$ if necessary (independent of $p_0,k,s$) there holds
\begin{align}\label{ine-s-BWsm-sing}
C^{-1} |s(\mathbf{z})|_{h^k}^2\le  |g_s(\mathbf{z})|^2= |s(\mathbf{z})|_{h^k}^2 e^{k (\psi(\mathbf{z})+ \phi(\mathbf{z}))} \le C \, M,
\end{align}
for $\mathbf{z} \in \B(0, k^{-2/\alpha})$. Let $\lambda':= \lambda/ (1+ \lambda)$. One also sees that there is a constant $C_0>0$ independent of $k,s,p_0$ such that 
\begin{align}\label{ine-s-BWsm-singbosung}
|g_s(0)|^2 \le |g_s(\mathbf{z})|^2+ M/C^{2/\lambda'}
\end{align}
if $|\mathbf{z}|\le 2 k^{-2/ \alpha}/C_0$. 
Since $\mu= \rho \Leb_K$, where $\rho^{-\lambda} \in L^1(\Leb_K)$, 
applying H\"older inequality to $|g_s|^{2\lambda'}=
(|g_s|^{2\lambda'} \rho^{\lambda'})(\rho^{-\lambda'})$ gives
\begin{align}\label{ine-chanchuanL1cuagssm-sing}
\int_{\{|\mathbf{z}| \le  k^{-2/\alpha}\}} 
|g_s|^{2\lambda'} d \Leb_K \lesssim   \bigg(\int_{\{|\mathbf{z}| 
\le  k^{-2/ \alpha}\}} |g_s|^2 d \mu \bigg)^{\lambda'} \lesssim 1.
\end{align}
Let $A \ge C^6$ be a constant.   
Now let $F: \D \times Y \to X$ be the family of analytic discs 
in Theorem \ref{th-discnoboundary} associated to $C_0$ 
for $p_0=0,$  and $t:= k^{-2/\alpha}$.

Let $g:= |g_s \circ f|^{2\lambda'}$. 
Put $\mathbf{z}_y:= F(1-1/A, y)$ which is $\le 2 k^{-2/ \alpha}/C_0$ 
by properties of $F$ if $A$ is big enough.  
Applying Lemma \ref{le-hamsubtrenduongtron} to 
$g(\cdot,y)$ and $\beta=1/2$ yield
$$
|g_s(\mathbf{z}_y)|^{2\lambda'}= g(\xi_y,y) \le  
C M^{\lambda'} /A^{1/2}  +  
A C \int_{-\theta_0}^{\theta_0} g(e^{i\theta}, y) d \theta
$$
(again we increase $C$ if necessary independently of $s,p_0,k$).
By this and (\ref{ine-s-BWsm-singbosung}), we gets
\begin{align*}
|g_s(0)|^{2\lambda'} &\le  M^{\lambda'}/C^{2}+
C M^{\lambda'} /A^{1/2}+
A C \int_{-\theta_0}^{\theta_0} g(e^{i\theta}, y) d \theta \\
&\le 2M^{\lambda'}/C^2+
AC\int_{-\theta_0}^{\theta_0} g(e^{i\theta}, y) d \theta.
\end{align*}
Integrating the last inequality over $y \in Y$ gives
\begin{align}\label{ine-uocluonggapdungfxismsin}
|g_s(0)|^{\lambda'} \le 2M^{\lambda'}/C^2+
A C (\vol(Y))^{-1} \int_Y vol_y \int_{-\theta_0}^{\theta_0} 
g(e^{i\theta},y) d \theta.
\end{align}
The  second term in the right-hand side of the last inequality is bounded by  
$$I_k:= \int_{\B(0, k^{-2/\alpha})\cap K}
|g_s(x)|^{2\lambda'} |\det DG|^{-1} d \Leb_K
\le k^{2n/\alpha} \int_{\B(0, k^{-2/\alpha})\cap K}
|g_s(x)|^{2\lambda'} d \Leb_K.$$
Using this, (\ref{ine-uocluonggapdungfxismsin}) and 
\eqref{ine-chanchuanL1cuagssm-sing} gives
$$\big(g_s(0)\big)^{2\lambda'}  \le  
2M^{\lambda'}/C^2 + A C_2 k^{2 n/\lambda'}$$
for some constant $C_2$ big enough (independent of $k,s,p_0$). 
This combined with (\ref{ine-s-BWsm-sing}) gives
$$|s(p_0)|_{h^k}^{2\lambda'}  \le  
2M^{\lambda'}/C + A C C_2 k^{2 n/\lambda'}$$
for every $p_0 \in K_k$. By this and \eqref{ine-chantrensxganKsing}
we obtain
$$M^{\lambda'}  \lesssim k^{2n/\alpha}$$
by choosing $C$ big enough.
Hence $M \lesssim  k^{2 n/(\alpha \lambda')}.$
 This finishes the proof. 
\end{proof}

We now explain how to  prove Theorem \ref{the-strongBMintro}
when $K$ is not necessarily totally real. 
\begin{proof}[End of proof of  Theorem \ref{the-strongBMintro} 
for $\dim K \ge n$]
As in the case of $\dim K=n$, it suffices to work with points in $K$ 
of distance at least $k^{-2/\alpha}/C_0$ to the singularity of $K$ 
for some $C_0>0$ big enough. Let $p_0$ be such a point, 
and let $s \in H^0(X,L^k)$ with $\|s\|_{L^2(\mu, h^k)}=1$. 
Our goal is to show that $|s(p_0)|_{h^k} \le 
C k^{2n_K(\lambda+1)/(\alpha \lambda)}$ 
for some constant $C$ independent of $s$ and $k$. 
We identify $s$ with a holomorphic function $g_s$
on a small open neighborhood $U$ of $p_0$ as usual. 
Hence as above one get
\begin{align}\label{ine-CRdimduong}
\int_{\B(p_0,k^{-2/\alpha})\cap K} |g_s|^{2 \lambda'}d \Leb_K \lesssim 1,
\end{align}
where $\lambda'$ is as in the case of $\dim K=n$.

Let $r$ be the CR dimension of $K$. Recall that $\dim K = n+r$. 
Using  standard local coordinates near $p_0$ on $K$, 
one sees that by shrinking $U$ if necessary, there are 
holomorphic local coordinates 
$(\mathbf{z}_1,\mathbf{z}_2) \in U \subset \C^{n-r}\times \C^{r}$
such that $K$ is locally given by the graph 
$\Im \mathbf{z}_1=h(\Re \mathbf{z}_1, \mathbf{z}_2)$, 
for $h \in \mathcal{C}^5$. In particular for every vector 
$v\in \R^{r}$ (small enough) the real linear subspace $\Im z_2= v$ 
intersects $K$ at a  generic  CR $\mathcal{C}^5$ smooth submanifold 
$K_v$ in $U$. Put $g_{s, v}:= g_s|_{K_v}$. Let $C_0>0$
be a big constant to be chosen later. By Fubini's theorem
and \eqref{ine-CRdimduong}, we obtain
$$\int_{|v| \le k^{-2/\alpha}} d \Leb_{\R^r}
\int_{K_v} |g_{s,v}|^{2 \lambda'} d \Leb_{K_v} \lesssim 1.$$  
It follows that there exists $v$ with $|v| \le k^{-2/ \alpha}/C_0$ 
so that 
$$\int_{K_v} |g_{s,v}|^{2 \lambda'} d \Leb_{K_v}
\lesssim k^{-2r/ \alpha} .$$
Applying the proof of the case where $\dim K=n$ to $K_v$, we see that 
\begin{align}\label{ine-nhatcatImz2}
|s(0, h(0,0,v), 0,v)|^2_{h^k} \lesssim k^{2n_K/(\alpha \lambda')},
\end{align}
where we write $(\mathbf{z}_1, \mathbf{z}_2)=
(\Re \mathbf{z}_1, \Im \mathbf{z}_1, \Re \mathbf{z}_2,
\Im \mathbf{z}_2)$, and $p_0$ is identified with $0$ in these local coordinates. 
On the other hand since $|v| \le k^{-2/\alpha}/C_0$, 
using a version of Bernstein-Markov inequality
(similar to \eqref{BMsection}) yields
$$|s(0)|^2_{h^k} \le |s(0, h(0,0,v), 0,v)|^2_{h^k}+
\frac{1}{2}\sup_{K} |s|^2_{h^k}$$
if $C_0$ is big enough.  This combined with (\ref{ine-nhatcatImz2})
gives the desired upper bound. The proof is finished.
\end{proof}

We end the subsection with a remark.
 
\begin{remark}\label{re-local-Bkbound} 
Let the hypothesis be as in Theorem \ref{the-strongBMintro}. 
In the proof of Theorem \ref{the-strongBMintro}, we actually 
proved the following local estimate.  Let $U$ be an open subset in $X$, 
and $\mu_U$ be the restriction of $\mu$ to $U$, define
$$B_{k,U}:= \sup_{\{s \in H^0(U,L^k)\}}
\frac{|s|^2_{h^k}}{\|s\|^2_{L^2(\mu_U,h^k)}} \cdot$$
Then for every $U' \Subset U$, there exists
$C'>0$ such that 
$$\sup_{K \cap U'} B_{k,U} \le
C' k^{2n_K(\lambda+1)/(\alpha \lambda)}.$$ 
\end{remark}

\section{Zeros of random polynomials}

In this section we prove Theorem \ref{the-zeros}.
Let $\Leb_{\C^m}$ be the Lebesgue measure on $\C^m$ for $m \ge 1$, and we denote by $\|\cdot\|$ the standard Euclidean norm on $\C^m$. Let $\omega_{FS,m}$ be the Fubini-Study form on $\P^m$, and let $\Omega_{FS, m}:= \omega_{FS,m}^m$ be the Fubini-Study volume form on $\P^m$. We always embed $\C^m$ in $\P^m$.  We recall the following key lemma. 

\begin{lemma}\label{le-maxkhithetich} (\cite[Proposition A.3 and Corollary A.5]{DS_tm}) There exist constants $C, \lambda>0$ such that for every $k \ge 0$, and every $\omega_{FS,k}$-psh function $u$ on $\P^k$ with $\int_X u \Omega_{FS,k}=0$, then 
$$u \le C(1+ \log k), \quad \int_{\{u<-t\}} \Omega_{FS,k} \le C k e^{- \lambda t}$$
for every $t \ge 0$.
\end{lemma}

The essential point is that the constants in the above result is uniformly in the dimension $k$ of $\P^k$. In our applications, the dimension $k$ will tend to $\infty$. Let $d_k$ be the dimension of the space of polynomials of degree at most $k$ on $\C^n$. Note that $d_k \approx k^n$.

  Let $f$ be a bounded Borel function on $\C$ such that there is a constant $C_0>0$ for which for every $r>1$ we have 
$$\int_{|z| \ge r} f d \Leb_\C \le C_0/ r^2.$$ 
 Let $a_1,a_2,\ldots, a_{d_k} $ be complex-valued i.i.d  random variable whose distribution is $f \Leb_\C$. Assume furthermore that
the joint-distribution of $(a_1,\ldots, a_{d_k})$ satisfies
$$\cali{P}_k:= f(z_1) \ldots f(z_{d_k}) \Leb(z_1) \otimes \cdots \otimes \Leb(z_{d_k}) \le C_0 \Omega_{FS,d_k} $$  on $\C^{d_k}$.

Let $p^{(d_k)}:= (p_1,\ldots, p_{d_k})$ be an orthonormal basis of $\mathcal{P}_k(\C^n)$.  Let $L$ be a complex algebraic subvariety of dimension $m \ge 1$ in $\C^n$. Note that the tolological closure of $L$ in $\P^n$ is an algebraic subvariety in $\P^n$. 
Observe that 
$$\omega_{FS,n}^m \le C \omega^m,$$
where $\omega$ is the standard K\"ahler form on $\C^n$, and $C>0$ is a constant.    Fix a compact $A$ of volume $\vol(A):=\int_A \omega_{FS,n}^m>0$ in $\C^n$. We start with a version of \cite[Lemma 2.4]{Bloom-Levenberg-random} with more or less the same proof. We use the Euclidean norm on $\C^{d_k}$.

\begin{lemma} \label{le-giongawn} Let $M \ge 1$ be a constant. Let $E_k$ be the set of $(a_1,\ldots, a_{d_k}) \in \C^{d_k}$ such that
$$\int_A \bigg(\log \bigg|\sum_{j=1}^{d_k} a_j p_j(z)\bigg| - 1/2\log \sum_{j=1}^{d_k} |p_j(z)|^2 \bigg) \omega_{FS,n}^m  \ge 2 M \vol(A)\log d_k.$$Let $E'_k$ be the set of $a^{(d_k)}$ so that $\|a^{(d_k)}\| \ge d_k^{2M}$.
Then we have $$\cali{P}_k(E_k \cup E_k') \le C d_k^{-3M}$$
for some constant $C>0$ independent of $k$ and $M$.
\end{lemma}

\proof Put $a^{(d_k)}:= (a_1,\ldots, a_{d_k})$, and   
$$I_k(a_1,\ldots, a_{d_k}):= \int_A \bigg(\log \bigg|\sum_{j=1}^{d_k} a_j p_j(z)\bigg| - 1/2\log \sum_{j=1}^{d_k} |p_j(z)|^2 \bigg) \omega_{FS,n}^m$$
Observe 
 $$\bigg|\sum_{j=1}^{d_k} a_j p_j(z)\bigg| \le \|a^{(d_k)}\| (\sum_{j=1}^{d_k} |p_j(z)|^2)^{1/2}.$$
 It follows that for $a^{(d_k)} \in E_{k}$, one has
 $$2M \vol(A) \log d_k \le I_k(a^{(d_k)}) \le \log \|a^{(d_k)}\|\vol(A).$$
 This implies that for each $k$ there exists $1\le j_k \le d_k$ such that $|a_{j_k}| \ge d_k^{(4M-1)/2}$. We infer that 
\begin{align*}
\cali{P}_k(E_{k}) &\le \cali{P}_k\bigg(a^{(d_k)}: |a_j| \ge d_k^{(4M-1)/2} \, \text{for some } 1\le j \le d_k\bigg) \\
&\le  d_k \int_{|a_1| \ge d_k^{(4M-1)/2}} f \Leb_C \le C_0 d_k^{-4M+1} \le C_0 d_k^{-3M}.
\end{align*}  
Similarly, we also get 
$\cali{P}_k(E'_k) \lesssim d_k^{-3M}$. This finishes the proof.
\endproof

Put $a^{(d_k)}:= (a_1,\ldots, a_{d_k})$, and $p^{(d_k)}:= (p_1,\ldots, p_{d_k})$.  Define
$$\varphi(a^{(d_k)}):= \int_{z \in L} \log \frac{\big| \sum_{j=1}^{d_k} a_j p_j(z)\big|}{(\|a^{(d_k)}\|^2+1)^{1/2}\|p^{(d_k)}(z)\|} \omega_{FS, n}^m(z).$$  
Observe that $\varphi \le 0$ on $\C^{d_k}$. We put
$$I_\varphi:= \int_{\C^{d_k}} \varphi(a^{(d_k)}) \Omega_{FS, d_k}$$

\begin{lemma} \label{le-uocluongtichphancuavarpgiNk} There exists a constant $C>0$ such that for every $k \ge 1$ we have
$$I_\varphi \ge - C \log d_k.$$
\end{lemma}

\proof  Let $I$ be the right-hand side of the desired inequality. By Funibi's theorem and the transitivity of the unitary group on $\C^{d_k}$, one has
\begin{align*}
I_\varphi &= \int_{z \in L} \omega_{FS, n}^m \int_{\C^{d_k}} \varphi(a^{(d_k)}) \Omega_{FS, d_k}\\
&= \int_{z \in L} \omega_{FS, n}^m \int_{\C^{d_k}} \log \frac{|a_1|}{(\|a\|^2+1)^{1/2}} \Omega_{FS, d_k}.
\end{align*}
The function 
$$\psi:= \log \frac{|a_1|}{(\|a^{(d_k)}\|^2+1)^{1/2}}$$
 is $\omega_{FS,d_k}$-psh on $\P^{d_k}$ (where $\omega_{FS, d_k}$ is the Fubini-Study form on $\P^{d_k}$). Let $\psi':= \psi - \int_{\C^{d_k}} \psi \Omega_{FS, d_k}$. Thus $\int_{\P^{d_k}} \psi' \Omega_{FS, d_k}=0$ and $\psi'$ is $\omega_{FS, d_k}$-psh. Applying Lemma \ref{le-maxkhithetich} to $\psi'$ gives
$$\psi' \le c(1+ \log d_k)$$
for some constant $c>0$ independent of $k, \psi'$. Consequently, 
$$\psi(a^{(d_k)}) \le c(1+ \log d_k)+ \int_{\C^{d_k}} \psi \Omega_{FS, d_k}$$
for every $a^{(d_k)} \in \C^m$. In particular for $a^{(d_k)}= (1,0, \ldots, 0)$, we obtain 
$$\int_{\C^{d_k}} \psi \Omega_{FS, d_k} \ge - c(1+ \log d_k)-1.$$ 
Thus the desired inequality follows.
\endproof

\begin{proof}[End of the proof of Theorems \ref{the-zeros} and \ref{the-zerosgiaoL}]
Let $\varphi':= \varphi -I_\varphi$. We have $\int_{\C^{d_k}} \varphi' \Omega_{FS, d_k}=0$. By Lemma \ref{le-maxkhithetich} again, there are constants $c, \alpha>0$ independent of $k$ such that  
$$\int_{\{\varphi' \le -t\}} \Omega_{FS, d_k} \le c d_k e^{-\alpha t}. $$
Combining this with Lemma \ref{le-uocluongtichphancuavarpgiNk} yields
$$\int_{\{\varphi \le - C \log d_k -t\}} \Omega_{FS, d_k} \le c d_k e^{-\alpha t}. $$
Let $M\ge 1$ be  a constant. Choosing $t:= 4M  \alpha^{-1} \log d_k$, where $C>0$ big enough  gives 

\begin{align}\label{le-chantrenvarphi}
\int_{\{\varphi \le - (C+3/\alpha) \log d_k\}} \Omega_{FS, d_k} \le c d_k^{-3M}.
\end{align}

Let $E'_k$ be the set of $a^{(d_k)}$ such that $\|a^{(d_k)}\| \le d_k^{2M}$ and $\varphi \ge -(C+4M/\alpha) \log d_k$. Combining (\ref{le-chantrenvarphi}) and Lemma \ref{le-giongawn}, we obtain that 
$$\cali{P}_k(\C^{d_k} \backslash E'_k) \lesssim  2c d_k^{-3M}$$
(we increase $c$ if necessary).
On the other hand, by the definition of $E'_k$ and $\varphi$, we see that the set of $a^{(d_k)}$ such that 
$$\int_{z \in L} \log \frac{\big| \sum_{j=1}^{d_k} a_j p_j(z)\big|}{\|p^{(d_k)}(z)\|} \omega_{FS, n}^m(z) \ge   -C'_M \log d_k$$
(for some  constant $C'_M>0$ big enough independent of $k$)
contains $E'_k$. It follows that 
$$\cali{P}_k\bigg\{ a^{(d_k)}: \int_{z \in L} \log \frac{\big| \sum_{j=1}^{d_k} a_j p_j(z)\big|}{\|p^{(d_k)}(z)\|} \Omega_{FS, n}(z) \le   -C'_M \log d_k  \bigg\} \le 2c d_k^{-3M}. $$ 
This together with Lemma \ref{le-giongawn} implies that there exists a Borel set $F_k$ such that $\cali{P}_k(F_k) \le 3c d_k^{-3M}$, and for $a^{(d_k)} \not \in F_k$, one has $\|a^{(d_k)}\| \le d_k^{2M}$, and 
$$\int_{z \in L} \log \frac{\big| \sum_{j=1}^{d_k} a_j p_j(x)\big|}{\|p^{(d_k)}(z)\|} \omega_{FS, n}^m(z) \ge   -C'_M \log d_k. $$ 
Consequently
for  $p= \sum_{j=1}^{d_k} a_j p_j$,  $\psi_k:= 1/(2k) \log \tilde{B}_k$ and  $a^{(d_k)} \not \in F_k$, there holds
$$- C'_M\log k/ k \le \int_{L} (k^{-1} \log |p^{(d_k)}| - \psi_k) \omega_{FS, n}^m.$$
 Moreover $\|a^{(d_k)}\| \le d_k^{2M}$, one has  
\begin{align*}
|k^{-1} \log |p|- \psi_k| &= 2  \max\{k^{-1} \log |p|, \psi_k\}- \psi_k -k^{-1}\log |p| \\
&\le  \psi_k -k^{-1}\log |p|+ 2 M k^{-1} \log d_k.
\end{align*}
It follows that 
$$\int_{L} | k^{-1} \log |p| - \psi_k| \omega_{FS, n}^m \le C''_M \log d_k/k,$$
for some constant $C''_M>0$ independent of $k$. This together with Theorem \ref{th-xapxiphiKBkCn} implies
$$\int_{L} | k^{-1} \log |p| -V_{K,Q}| \omega_{FS,n}^m \le C''_M \frac{\log k}{k}$$
by increasing $C''_M$ if necessary.
It follows that 
$$\dist_{-2} \big(k^{-1}[p=0] \wedge [L], \ddc V_{K,Q} \wedge [L]\big) \lesssim \int_{L} | k^{-1} \log |p| - \psi_k| \omega_{FS, n}^m \le  \frac{C''_M\log k}{k},$$
for $a^{(d_k)} \not \in F_k$. Here recall that we define $\dist_{-2}$ by considering $[p=0]\wedge [L]$ and $\ddc V_{K,Q} \wedge [L]$ as currents on $\P^m$. This finishes the proof.
\end{proof}

\bibliography{biblio_family_MA,biblio_Viet_papers}
\bibliographystyle{siam}

\bigskip

\noindent
\Addresses
\end{document}